\documentclass{article}    

\usepackage{graphicx}
\usepackage{amsfonts,amssymb,amsbsy,amsmath,mathrsfs,enumerate,verbatim,amsthm}
\usepackage{mathptmx}
\usepackage{helvet}
\usepackage{courier}        
\usepackage{type1cm}
\usepackage{makeidx}         
\usepackage{multicol}
\usepackage[bottom]{footmisc}
\usepackage{epsf}
\usepackage{epstopdf}
\usepackage{pdfsync}
\usepackage{epsfig}
\usepackage{algorithmic,algorithm}
\usepackage{hyperref}
\usepackage{color}  
\usepackage{afterpage}
\usepackage{verbatim}

\makeindex             

\newtheorem{theorem}{Theorem}[section]
\newtheorem{corollary}{Corollary}[section]

\newcommand{\norm}[1]{\left\Vert#1\right\Vert}
\newcommand{\abs}[1]{\left\vert#1\right\vert}
\newcommand{\p}{\partial}

\allowdisplaybreaks

\begin{document}

\title{Energy norm error estimates  and convergence analysis for a  stabilized  
     Maxwell's equations in conductive media}

\author{E. Lindström \thanks{Eric Lindström, Email: erilinds@chalmers.se} \and L. Beilina \thanks{Larisa Beilina, Email: larisa@chalmers.se \\ Department of Mathematical Sciences, Chalmers University of Technology and University of Gothenburg, SE-412 96 Gothenburg Sweden}}

\date{\today}
\maketitle

\graphicspath{% directories to search the graphics files
  {FIGURES/} 
}

%{\color{red}
\begin{abstract}
  The aim of this article is to investigate the well-posedness,
  stability and convergence of solutions to the time-dependent
  Maxwell's equations for electric field in conductive media in
  continuous and discrete settings. The situation we consider would
  represent a physical problem where a subdomain is emerged in a
  homogeneous medium, characterized by constant dielectric
  permittivity and conductivity functions. It is well known that in
  these homogeneous regions the solution to the Maxwell's equations
  also solves the wave equation which makes calculations very
  efficient. In this way our problem can be considered as a coupling
  problem for which  we derive stability and convergence analysis. A number
  of numerical examples validate theoretical convergence rates of the
  proposed stabilized explicit finite element scheme.
\end{abstract}
%}



\section{Introduction}

%{\color{red}
In this paper we consider the time-dependent Maxwell's equations in a
bounded, simply connected spatial domain $\Omega$. This domain is
divided into two subdomains, one outer were the dielectric
permittivity and conductivity are constant functions, and one inner
one were they are allowed to vary but are still bounded functions.

One important (and of special interest for the authors) consequence of
the results developed in this paper are applications of Maxwell's
equations to solutions of \textit{Coefficient Inverse Problems}
(CIPs). In \cite{TBKF1,TBKF2} one can read about inverse
problems applied to imaging of buried objects, and in
\cite{BL1,BL2,BondestaB} inverse problems are used for medical imaging. In
the latter case the problem is to reconstruct the dielectric
permittivity and conductivity functions of an anatomically realistic
phantom of a breast tissue.
%Within a domain of interest of the breast,
The dielectric properties of different tissue types in a breast are
experimentally measured and known \cite{dielprop}, but their
distribution inside every particular breast tissue is unknown.  Such
a scenario is one case where the domain decomposition can be an useful
tool for solution of electromagnetic CIPs when the goal is
determination of dielectric properties of the object from boundary measurements of
the scattered electric field.

%The domain decomposition is also beneficial for calculations as
%mentioned earlier, and theoretically simpler.
Under certain
circumstances, it is known that the solution to the Maxwell's
equations also solves the wave equation, which is more studied and
understood, see \cite{AB1,AB2,BMaxwell,BL1,BR1}.
In \cite{BR1}, a finite element analysis shows stability
  and consistency of the stabilized  finite element method for the
  solution of  Maxwell's equations in non-conductive media, and
in \cite{AB1} authors investigated  a stabilized  domain decomposition
finite element method for the time
harmonic Maxwell's equations. Stability and convergence analysis of a Domain Decomposition FE/FD
 method for  time-dependent Maxwell’s equations  was presented in \cite{AB2}.
In our knowledge, all previous cited
works consider non-conductive media, and the research concerning
time-dependent Maxwell's equation for electric field in conductive media, when both dielectric
permittivity and conductivity are space-dependent functions, are
missing.
%somewhat lacking as far as the authors know.

The stability and well-posedness of the wave equation are well
understood and studied \cite{evans,joly}. Certain model of wave equations
have also been used to model inverse problems, see
\cite{BK,carlemanwaves,GG}. The progression to Maxwell's equations is
arguably natural, since the system has wave-like properties. However,
some complications occur from the presence of the double curl
operator.
%We have already mentioned how spurious solution can occur in
%implementations, but
Another theoretical complication is that when one analyzes the
corresponding bilinear form induced by the variational form, one can
see that it is not coercive. This coercivity is often critical in
proofs concerning existence and uniqueness of solutions.  The
additional novelty of the presented work is in how we deal with
coercivity of the bilinear form. Since the bilinear form with presence
of time-dependent terms is non-coercive, we split it and separate
terms with derivatives in time in order to derive the coercivity for
the remaining spatial part of the bilinear form using some minor
restrictions on the gradient of the permittivity function.  We use
then coercivity of the spatial part of the bilinear form in the proof
of a priori error estimate. Derivation of coercivity  of the entire scheme is a topic of an ongoing research.

To align the results with implementations of $P1$-methods, we
introduce a slightly altered, stabilized problem. Otherwise, these
methods can lead to spurious solutions (see
\cite{Assous,BMaxwell,Ciarlet,Ciarlet_Jamelot,E_Jamelot,Jiang1,Jiang2,Jin,div_cor,PL})
and is well-known that theoretically, divergence free edge edge
elements are a better fit, see
\cite{Cohen,Ern_Guermond2,Monk,MP,Nedelec}. To read more about various
  numerical methods for Maxwell equations and more details about the
complications, see
\cite{Arnold_Brezzi,Assous,AB1,Bonnet_etal,Dauge_Cost1,Dauge_Cost2,Ern_Guermond2,delta,Monk,div_cor}
and references therein.
Naturally, since the theoretical results of this work have importance
for numerical implementations, we also present analysis of the
corresponding discrete problem to our original pmodel.

An outline of this paper is as follows. 
In Section \ref{sec:model} we introduce the mathematical model and present the  stabilized
   problem 
   for the time-dependent Maxwell's equations  in conductive media.
   In Section  \ref{sec:fem} we state the variational problem 
  for the stabilized model and 
  formulate the finite element scheme.
  Section \ref{sec:energynorm} is devoted to the  energy norm error analysis
  and section \ref{sec:apriori}  presents derivation of  a priori error estimates.
In Section   \ref{sec:numex} are  performed
numerical   convergence tests illustrating theoretical results of  this paper.
Finally, in Section  \ref{sec:conclusion} we conclude the results of the paper.

\section{The mathematical model}

\label{sec:model}

Let us consider
the initial value problem for the 
  electric field $E\left( x,t\right)=\left( E_{1},E_{2},E_{3}\right) \left( x,t\right)$,  
  $x \in\mathbb{R}^{3}$,  $t \in [0, T]$, for time-dependent
   Maxwell's equations  in conductive media, under the 
assumptions  that the dimensionless relative magnetic permeability of 
the medium is $\mu_r \equiv 1$:
\begin{equation}\label{E_gauge1}
\begin{split}
 % \varepsilon(x) \frac{\partial^2 E}{\partial t^2} +  \nabla ( \nabla \cdot E) - \nabla \cdot (\nabla E)   &= -\sigma(x) \frac{\partial E}{\partial t} ~ \mbox{in}~~ \Omega_T,
 \frac{1}{c^2} \varepsilon_r(x) \frac{\partial^2 E}{\partial t^2} +  \nabla \times \nabla \times E  &= - \mu_0 \sigma(x) 
 \frac{\partial E}{\partial t} -j, \\
  %~ \mbox{in}~~ \Omega \times (0,T],  \\
  \nabla \cdot(\varepsilon E) &= 0, \\
  E(x,0) = f_0(x), ~~~\frac{\partial E}{\partial t}(x,0) &= f_1(x) ,~~ x \in \mathbb{R}^{3},t\in (0,T].
\end{split}
\end{equation}
Here, $\varepsilon_r(x) = \varepsilon(x)/\varepsilon_0$
is the dimensionless relative dielectric permittivity, 
$\sigma(x)$ is the
electric conductivity function; $\varepsilon_0$, and 
$\mu_0$ are the permittivity and permeability of the free space,
respectively, and $c=1/\sqrt{\varepsilon_0 \mu_0}$ is the speed of
light in free space and $j$ is a given source function.

To solve the problem (\ref{E_gauge1}) numerically, we consider it in a
bounded simply connected space domain $ \Omega\subset \mathbb{R}^{n}, \, n=2,3$ with
boundary $\Gamma$ and time domain $J = (0,T)$.
In this work we will study the problem (\ref{E_gauge1})  in a special framework:  we
decompose the space domain $ \Omega$ into two subdomains such that
$\Omega = \Omega_{\rm 1} \cup \Omega_{\rm 2}, \Omega_{\rm 1}
\subset \Omega$  and $ \Omega_{\rm 1} =  \Omega_{\rm IN} \cup \Omega_{\rm OUT} $.
We   assume that for some known constants $d_1 > 1, d_2 > 0$  chosen such that $d_1 > d_2$, the functions
$\varepsilon, \sigma \in C^{2}\left( \Omega \right)$ 
% of equation (\ref{E_gauge}) 
satisfy  following conditions:
\begin{equation} \label{2.3}
%\left\{  
\begin{split}
  \varepsilon_r(x) &\in \left[ 1,d_1\right],\quad
  \sigma(x) \in \left[0,d_2\right], ~\text{ for }x\in  \Omega _{\rm IN}, \\
  ~  ~ \varepsilon_r(x) &=1, \quad \sigma(x) = 0\quad  \text{ for }x\in  \Omega _{\rm 2} \cup \Omega _{\rm OUT}.
\end{split}
\end{equation}
We refer to \cite{BL1} and references therein for  justification and  possible choice of these  coefficients.

We observe that conditions \eqref{2.3} on $\varepsilon$ and $\sigma$ 
together with the relation 
\begin{equation}\label{divfree}
\nabla \times \nabla \times E = \nabla (\nabla
\cdot E) - \nabla \cdot ( \nabla E), 
\end{equation}
and divergence free condition $\nabla \cdot (\varepsilon E) = 0$ in $\Omega_{\rm 2}$,
make equations in 
(\ref{E_gauge1}) independent of each others in  $\Omega _{\rm 2}$ such that
in  $\Omega _{\rm 2}$,  
 we  solve the system of  uncoupled  wave equations: 
\begin{equation} \label{6.11}
  \frac{\partial^2 E}{\partial t^2}-\Delta E= 0, ~~~(x,t) \in  \Omega_{\rm 2} \times ( 0,T]. 
\end{equation}

In \cite{BR1}, a finite element analysis shows stability
  and consistency of the stabilized  finite element method for the
  solution of (\ref{E_gauge1}) with $\sigma(x)=0$.
In \cite{AB1}   a stabilized linear, domain decomposition
finite element method for the time
harmonic Maxwell's equations  was studied.
  In the current study
  we show stability and convergence analysis of the finite element method for  solution
   of (\ref{E_gauge1})
  under the condition \eqref{2.3} on $\varepsilon$  and $\sigma$.

  Let $\Omega_T = \Omega \times (0,T), \Gamma_T = \Gamma \times (0,T)  $.
Let us
introduce the following spaces of real valued  functions
\begin{equation}\label{spaces}
  \begin{split}
    H_E^2(\Omega_T) &:= \{  w \in H^2(\Omega_T):  w( \cdot , 0) = f_0,  \frac{\partial w}{\partial t}( \cdot , 0) = f_1 \}, \\
H_E^1(\Omega_T) &:= \{  w \in H^1(\Omega_T): w( \cdot , 0) = f_0,  \frac{\partial w}{\partial t}( \cdot , 0) = f_1\}.
\end{split}
\end{equation}
In this paper we study the following  stabilized initial boundary value problem  setting
${\textbf H_E^2(\Omega_T)} =  [H_E^2(\Omega_T)]^3$:  find
$E \in  {\textbf H_E^2(\Omega_T)}$  such that
\begin{equation}\label{eq1} 
\left\{
\begin{array}{ll}
    \varepsilon \frac{\partial^2 E}{\partial t^2}  - \triangle E - \nabla (\nabla \cdot ((\varepsilon -1)E)) =  - \sigma(x) 
 \frac{\partial E}{\partial t} -j& \mbox{ in } \Omega_T, \\
    E(\cdot,0) = f_0(\cdot), \mbox{ and } \partial_t E (\cdot,0) = f_1(\cdot) & \mbox{ in } \Omega, \\
    % \partial _{n} E = - \partial_t E & \mbox{ on } \partial \Omega \times (0,T), \\
  %  \frac{\partial E}{\partial n} = 0 & \mbox{ on } \partial \Gamma \times (0,T). \\
     E = 0 & \mbox{ on } \Gamma_T. \\  
 %   \nabla \cdot (\varepsilon E) = 0 & \mbox{ in } \Omega.
  \end{array}
  \right. 
\end{equation}

Here, the divergence free condition $\nabla \cdot(\varepsilon E)=0$ is hidden in the first equation of system \eqref{eq1}.

\section{Finite Element Discretization}

\label{sec:fem}

Throughout the paper
we denote the inner product in space
of $[L^2(\Omega)]^d, d \in \{1,2,3\},$
by $(\cdot,\cdot)$, and the corresponding norm by $\parallel \cdot
\parallel$.

Let us  define the following $L_2$ scalar products used in the analysis: 
\begin{equation}
  \begin{split}
    (u,v) &:=    (u,v)_\Omega  = \int_{\Omega} u  v~ d{\bf x},\quad
    ((u,v)) :=  ((u,v))_{\Omega_T} =  \int_0^T \int_{\Omega} u  v~ d{\bf x} dt,\quad  
 %   (u,v)_{\omega} := \int_{\Omega} \omega ~u  v d{\bf x}, \quad
    \\
    \langle u, v \rangle &:=   \langle u, v \rangle_\Gamma =  \int_\Gamma  u  v~ d{\sigma},\quad
    \langle\langle u, v \rangle\rangle := \langle\langle u, v \rangle\rangle_{\Gamma_T} = \int_0^T  \int_\Gamma  u v~ d{\sigma} dt.
  \end{split}
  \end{equation}
Additionally, we define   the $\omega$-weighted $L^2(\Omega)$ norm 
\begin{equation}
  \begin{split}
\| u \|_{\omega}:=\sqrt{\int_{\Omega} \omega | u|^2 \, d{\bf x}},\qquad 
 ~\omega > 0, \quad \omega \in L^{\infty}(\Omega)
\end{split}
\end{equation}
together with the $\omega$-weighted $L_2$  scalar product:
\begin{equation}
  \begin{split}
 (u,v)_{\omega} := \int_{\Omega} \omega ~u  v d{\bf x}.
\end{split}
\end{equation}

To  write  finite element scheme to solve the model  problem  \eqref{eq1}  in whole 
$\Omega$,
we discretize $\Omega_T = \Omega \times (0,T)$
by  partition $K_h = \{K\}$  of $\Omega$ into elements $K$, where 
$h=h(x)$ is a mesh function
defined as $h =  \max_{K \in K_h} h_K$. Here, where $h_K$
denotes the local diameter of the element $K$.
We also denote by  $\partial K_h = \{\partial K\}$ a partition of
 the boundary $\Gamma$ into boundaries  $\partial K$ of the elements $K$. Let $J_{\tau}$ be a uniform partition of the time interval $(0,T)$ into $N$ 
  equidistance subintervals $J=(t_{k-1},t_k]$ with the time step $\tau=T/N.$
%  such that
% $\tau = t_k - t_{k-1}$.
 We also 
 assume a minimal angle condition on elements $K$ in $K_h$ \cite{MA,KN}.

 To formulate 
the 
finite element method
for the spatial semi-discrete problem \eqref{eq1} 
 in $\Omega$
we introduce the finite element space $W_h^E(\Omega)$ for every
component of the electric  field $E$ defined by
\begin{equation}
W_h^E(\Omega) := \{ w \in H^1(\Omega): w|_{K} \in  P_1(K),   \forall K \in K_h \}, \nonumber
\end{equation}
where $P_1(K)$ denote the set of piecewise-linear functions on $K$.
We define ${f_0}_h, {f_1}_h, j_h$ to be the 
 usual interpolants of 
 $f_0, f_1, j$, respectively,
 in \eqref{eq1}  onto  $[W_h^E(\Omega)]^3$.

  Setting ${\textbf W^E(\Omega)} =  [W^E(\Omega)]^3$ and
${\textbf W_h^E(\Omega)} =  [W_h^E(\Omega)]^3$
   where the test function space  is chosen as
\begin{equation}\label{TestSpace}
{\textbf  W}_{h,0}^E(\Omega):=\{{\mathbf v}\in {\textbf  W_h^E}(\Omega)~ |~ {\mathbf v}=0 \quad  \mbox{on }\quad \Gamma\},
\end{equation}
the spatial semi-discrete problem \eqref{eq1} 
 in $\Omega$ 
  reads: 

\emph{Find } $E_{h} \in {\mathbf W_h^E(\Omega)}$ \emph{ such  that } 
$\forall {\mathbf v} \in {\mathbf  W}_{h,0}^E(\Omega)$, 
\begin{equation}\label{eq6}
  \begin{split}
    B(E_h, \mathbf{v}) := \left ( \varepsilon \partial_{tt} E_{h}, {\textbf v} \right ) +
    \left ( \sigma \partial_{t} E_{h}, {\textbf v} \right )
    + a( E_{h},  {\textbf v} )  &= -(j, {\textbf v}  ),  \\
\qquad   E_h(\cdot,0) &= {f_0}_h(\cdot), \\
\partial_t E_h(\cdot,0) &= {f_1}_h(\cdot) \quad 
   \mbox{ in } \Omega.  
  \end{split}
\end{equation}

Here, $a$ is a bilinear form  defined as

\begin{equation}\label{ah1}
  \begin{split}
  a( E_{h},  {\textbf v} ) &:= (\nabla E_h,\nabla {\textbf v})+ 
    (\nabla \cdot ((\varepsilon -1) E_h), \nabla \cdot {\textbf v} ) \\
&-( n \cdot \nabla \cdot ((\varepsilon -1) E_h), \nabla \cdot {\textbf v})_{\partial \Omega}  - \langle\partial_n E_{h}, {\textbf v}\rangle_{\Gamma}
 \end{split}
\end{equation}

We observe that boundary terms in \eqref{ah1}  disappear because of definition of test space \eqref{TestSpace}.
Thus, the bilinear form  \eqref{ah1} for the case of test space  \eqref{TestSpace}  will be transformed to

\begin{equation}\label{ah2}
  \begin{split}
  a( E_{h},  {\textbf v} ):= (\nabla E_h,\nabla {\textbf v})+ 
    (\nabla \cdot ((\varepsilon -1) E_h), \nabla \cdot {\textbf v} ).
 \end{split}
\end{equation}

Let us recall the  explicit 
  fully discrete finite element scheme for solution of \eqref{eq6}
 for $k=1,2,\ldots,N-1$
 and  $\forall {\mathbf v} \in {\mathbf  W}_{h,0}^E(\Omega)$
  which was derived in \cite{BL2}:

\begin{equation}\label{forwfem3}
  \begin{array}{l}
   \left ( {\varepsilon}_h
   \frac{E_{h}^{k+1} - 2 E_h^k + E_h^{k-1}}{\tau^2},    {\mathbf v}  \right )
   + (\nabla E_h^k,\nabla  {\mathbf v} )+
     (\nabla \cdot ({\varepsilon}_h E_h^k), \nabla \cdot   {\mathbf v}  ) 
    - (\nabla \cdot E_h^k, \nabla \cdot   {\mathbf v} ) \\
  %  + (g_h^k,  \bar{\lambda} )_{\partial \Omega_{\rm FEM}}
    +  (\sigma_h \frac{E_{h}^{k+1} - E_h^{k-1}}{ 2\tau},  {\mathbf v} )
    + (j_h^k,  {\mathbf v} )
    = 0,  \\
   {E_h}^0 = {f_0}_h \mbox{ and } {E_h}^1 =  {E_h}^0 + \tau {f_1}_h \mbox{ in } \Omega.  
  \end{array}
\end{equation}
In the scheme 
 \eqref{forwfem3}   we approximated $E_h(k\tau)$ and  $j_h(k\tau)$
 by $E_h^k$  and  $j_h^k$, respectively,  for $k=1,2,...,N$.
Rearranging terms in \eqref{forwfem3} we get for $k=1,2,\ldots,N-1$
and   $\forall {\mathbf v} \in {\mathbf  W}_{h,0}^E(\Omega)$
\begin{equation}\label{forwfem5}
  \begin{split}
 &   \left ( (\varepsilon_h  + \frac{\tau}{2}  \sigma_h ) E_h^{k+1},  {\mathbf v} \right)
    =   \left(2 \varepsilon_h  E_h^k,   {\mathbf v} \right) - \left( \varepsilon_h  E_h^{k-1},  {\mathbf v} \right)
    - \tau^2 ( \nabla E_h^k, \nabla    {\mathbf v})\\
 & - \tau^2   ( \nabla \cdot ({\varepsilon}_h E_h^k), \nabla \cdot   {\mathbf v}) 
+   \tau^2  ( \nabla \cdot E_h^k, \nabla \cdot  {\mathbf v} ) \\
&
+ \tau (\frac{\sigma_h}{2}  E_h^{k-1},  {\mathbf v} )   -  \tau^2 (j_h^k, v), \\
& E_h^0 = {f_0}_h \mbox{ and } E_h^1 =
          E_h^0 + \tau {f_1}_h \mbox{ in } \Omega.
  \end{split}
\end{equation}
  
  For the convergence of this scheme the following CFL condition derived in \cite{BR1} for the case of $\sigma=0$   should hold:
  \begin{equation}\label{CFL}
\tau \leq \frac{h}{\eta}, \eta = C \sqrt{ 1 + 3 \| \varepsilon -1\|_\infty},
  \end{equation}
  where $C$ is a mesh independent constant.
  The CFL condition for
  the case when both functions $\varepsilon \neq 0, \sigma \neq 0$ is topic of ongoing research.

\section{Energy norm error estimate (stability estimate)}

\label{sec:energynorm}

 In this section  first we  give a proof of energy estimate,
 for the vector $E\in H^{2}\left( \Omega_{T}\right) $ of
 the continuous model problem \eqref{eq1}. Then we formulate
 stability estimate for semi-discrete problem which is consequence of the energy
 estimate for the continuous problem.

%\subsection{Stability of the continuous problem}

 \begin{theorem}
 
 Assume that condition
  (\ref{2.3}) on the functions $\varepsilon(x), \sigma(x)$ hold. Let $\Omega \subset
\mathbb{R}^{3}$ be a bounded domain with the piecewise smooth
  boundary $\partial \Omega$. For any $ t^*\in \left( 0,T\right)
$ let $\Omega_{t^*}= \Omega\times \left( 0, t^* \right)$ and $\Gamma_{t^*} = \Gamma \times (0, t^*).$ Suppose that
  there exists a solution $E\in H^{2}\left( \Omega_{T}\right) $ of
  the model problem \eqref{eq1}.
  Then the vector $E$ is unique and there exists a constant
  $C = C(\|\varepsilon\|_{C^2(\Omega)}, \| \sigma\|, t^*)$
   such  that
   the following energy estimate is true  for all $\varepsilon \geq 1$ in
     \eqref{eq1}:
 \end{theorem}

\begin{equation}\label{energyestimate}
\begin{split}
\vert\vert\vert E \vert\vert \vert^2 (t^*) := & \left\Vert \partial _{t}E \right\Vert_{\varepsilon}^{2}(t^*) + \left\Vert E \right\Vert_{\sigma}^{2}(t^*) + \left\Vert \nabla E \right\Vert^{2}(t^*) + \left\Vert \nabla \cdot E \right\Vert _{\varepsilon - 1}^{2}(t^*) \\
&\leq C \left[ \left\Vert j\right\Vert _{\Omega _{t^*}}^{2}+\left\Vert f_{1} \right\Vert _{\varepsilon}^{2} + \left\Vert \nabla f_{0}\right\Vert ^{2} + \left\Vert f_{0}\right\Vert _{\sigma}^{2} + \left\Vert \nabla \cdot f_{0}\right\Vert_{\varepsilon - 1}^{2}\right] .
\end{split}
\end{equation}

\begin{proof}
We mark the terms in the first line of \eqref{eq1} with 
\begin{equation}\label{eq1terms}
    \underbrace{\varepsilon \partial_{tt} E}_{\textbf{I}_1}  \underbrace{-\Delta E}_{\textbf{I}_2} + \underbrace{\nabla ( \nabla \cdot (( 1 - \varepsilon ) E )}_{\textbf{I}_3} =  \underbrace{-\sigma \partial_t E}_{\textbf{I}_4} - \underbrace{j}_{\textbf{I}_5}
\end{equation}
to simplify notation in our proof, such that $\mathbf{I}_1 + \mathbf{I}_2 + \mathbf{I}_3 = \mathbf{I}_4 + \mathbf{I}_5$. 

Through the proof we denote a generic constant of moderate size by $C := C(\|\varepsilon-1\|_{\rm C^2(\Omega)}, \|\sigma\|_\infty, t^*)$. To prove our energy estimate we multiply \eqref{eq1terms} by $2\partial_t E$ and integrate over $\Omega_{t^*}$, and study it term by term. For the first term in \eqref{eq1terms} we have
\begin{align}\label{eq1term1}
    ((\mathbf{I}_1, 2 \partial_t E))_{\Omega_{t^*}} &= (( 2\partial_{tt} E, \partial_t E))_{\varepsilon, {\Omega _{t^*}}} \nonumber\\
    &= \int_0^{t^*} \partial_t \| \p_t E \|_\varepsilon^2(t) \, dt \nonumber\\
    &= \| \p_t E \|_\varepsilon^2(t^*) - \| f_1 \|_\varepsilon^2
\end{align}
where we used the chain rule, fundamental theorem of calculus and boundary conditions of \eqref{eq1}. 

Next we have that
\begin{align}\label{eq1term2}
    ((\mathbf{I}_2, 2 \partial_t E))_{\Omega_{t^*}} &= -((\Delta E,2 \partial_t E))_{\Omega_{t^*}} \nonumber\\
    &= -2 \langle\langle \partial_n E, \partial_t E\rangle\rangle_{\Gamma_{t^*}} + ((2\nabla E, \nabla (\partial_t E)))_{\Omega_{t^*}} \nonumber\\
    &= \int^{t^*}_0 \partial_t \| \nabla E \|^2(t) \, dt \nonumber\\
    &= \| \nabla E \|^2(t^*) - \| \nabla f_0 \|^2
\end{align}
where we use spatial integration by part, boundary conditions and that $\| \partial_n E\|_{\Gamma_{t^*}} = 0$.  

For the third term in \eqref{eq1terms} we integrate by parts spatially twice and get
\begin{align}\label{eq1term3}
    (( \mathbf{I}_3, 2 \partial_t E ))_{\Omega_{t^*}} &= (( \nabla ( \nabla \cdot (( 1 - \varepsilon ) E) , 2 \partial_t E ))_{\Omega_{t^*}} \nonumber\\
    &= 2 \langle \langle n \cdot ( \nabla \cdot (( 1 - \varepsilon ) E), \p_t E \rangle \rangle_{\Gamma_{t^*}} - 2 (( \nabla \cdot (( 1 - \varepsilon ) E), \nabla \cdot \p_t E))_{\Omega_{t^*}} \nonumber \\
    &= -2 (( \nabla (1 - \varepsilon) \cdot E, \nabla \cdot \p_t E ))_{\Omega_{t^*}} + ((2 \nabla \cdot E, \nabla \cdot \p_t E ))_{\varepsilon - 1, \Omega_{t^*}} \nonumber \\
    &= -2 \langle\langle \nabla(1 - \varepsilon) \cdot E, n \cdot \p_t E \rangle\rangle_{\Gamma_{t^*}} + 2 (( \nabla(\nabla(1 - \varepsilon) \cdot E), \p_t E))_{\Omega_{t^*}} \nonumber \\
    &+ \int_0^{t^*} \p_t \| \nabla \cdot E \|_{\varepsilon - 1}^2 (t) \, dt \nonumber\\
    &= 2 (( \nabla(\nabla(1 - \varepsilon) \cdot E), \p_t E))_{\Omega_{t^*}} + \| \nabla \cdot E \|_{\varepsilon - 1}^2 ({t^*}) - \| \nabla \cdot f_0 \|_{\varepsilon - 1}^2
\end{align}
Above we made use of \eqref{2.3} (note that since $\varepsilon \equiv 1$ on a neighbourhood of $\Gamma$, we have $(1 - \varepsilon)|_\Gamma \equiv 0$ and $\nabla(1 - \varepsilon)|_\Gamma \equiv 0$). We also made use of the fact that $\nabla \cdot ((1 - \varepsilon) E) = \nabla (1 - \varepsilon)E + (1 - \varepsilon)\nabla \cdot E$.

Before estimating the fourth term in \eqref{eq1terms} we first note that 
\begin{equation*}
    E(x,t^*) = E(x,0) + \int_0^{t^*} \partial_t E(x,t) \, dt,
\end{equation*}
and using that $(a+b)^2 \leq 2a^2 + 2b^2$ we get
\begin{align*} 
    &E^2(x,t^*) \leq 2 E^2(x,0) + 2(\int_0^{t^*} \partial_t E (x,t) \, dt)^2 \leq 2 E^2(x,0) + 2\int_0^{t^*} (\partial_t E)^2(x,t) \, dt.
\end{align*}
If we then integrate these terms over $\Omega$ we arrive at
\begin{align*}
    - 2 (( \p_t E, \p_t E ))_{\Omega_{t^*}}  \leq - \|E\|^2(t^*) + 2 \|E\|^2(0) = \|E\|^2(t^*) + 2 \|f_0\|^2.
\end{align*}
Applying this to our case we have that
\begin{align}\label{eq1term4}
    ((\mathbf{I}_4, 2 \p_t E ))_{\Omega_{t^*}} &= - 2 (( \p_t E, \p_t E ))_{\sigma, \Omega_{t^*}} \nonumber\\
    &\leq - \| E \|^2_\sigma (t^*) + 2 \|f_0\|^2_\sigma 
\end{align}

For our final term in \eqref{eq1terms}, we simply use that $2ab \leq a^2 + b^2$:
\begin{align}\label{eq1term5}
    ((\mathbf{I}_5, 2 \p_t E ))_{\Omega_{t^*}} &\leq ((|j|, 2|\p_t E| ))_{\Omega_{t^*}} \leq \| j \|^2_{\Omega_{t^*}} + \| \p_t E \|^2_{\Omega_{t^*}}
\end{align}

Next we collect all the terms \eqref{eq1term1}-\eqref{eq1term5} to arrive at
\begin{align}\label{collectedterms}
    &\|\partial_t E \|_\varepsilon^2(t^*) + \| \nabla E \|^2(t^*) + \| \nabla \cdot E \|^2_{\varepsilon - 1}(t^*) + \| E \|_\sigma^2 (t^*), \nonumber \\
    &\leq \| f_1 \|_\varepsilon^2 + \|\nabla f_0\|^2 + \| \nabla \cdot f_0\|_{\varepsilon - 1}^2 + 2 \| f_0 \|_\sigma^2 + \|j \|_{\Omega_{t^*}}^2 + \|\partial_t E\|_{\Omega_{t^*}}^2 \nonumber \\
    &+ 2 (( \nabla(\nabla(\varepsilon-1) \cdot E ), \partial_t E ))_{\Omega_{t^*}}. 
\end{align}
We may estimate the last term of \eqref{collectedterms} using that 
\begin{equation*}
    |\nabla(\nabla(\varepsilon - 1) \cdot E)| \leq C (|E| + |\nabla E|). 
\end{equation*}
Thus, the above estimate together with the inequality $ab \leq \frac{a^2}{2} + \frac{b^2}{2}$ yields
\begin{align*}
    2 (( |\nabla(\nabla(\varepsilon-1) \cdot E )| , |\partial_t E| ))_{\Omega_{t^*}} &\leq 2 C (( |E| + |\nabla E| , |\partial_t E| ))_{\Omega_{t^*}} \\
    &\leq C (\|(| E| + |\nabla E|)\|_{\Omega_{t^*}}^2 + \| \partial_t E\|_{\Omega_{t^*}}^2) \\ 
    &\leq C (\|E\|_{\Omega_{t^*}}^2 + \|\nabla E\|_{\Omega_{t^*}}^2 + \|\partial_t  E\|_{\Omega_{t^*}}^2). \\
\end{align*}
Now we can rewrite \eqref{collectedterms} as
\begin{align*}
    F(t^*) &\leq g(t^*) + C (\|\p_t E\|_{\Omega_{t^*}}^2 + \|\nabla E\|_{\Omega_{t^*}}^2 + \| E\|_{\Omega_{t^*}}^2) \\
    &\leq g(t^*) + C (\|\p_t E\|_{\varepsilon, \Omega_{t^*}}^2 + \|\nabla E\|_{\Omega_{t^*}}^2 +\|\nabla \cdot E \|^2_{\varepsilon - 1, \Omega_{t^*}} + \|E\|_{\sigma, \Omega_{t^*}}^2) \\
    & = g(t^*) + C\int_0^{t^*} F(t) \, dt,
\end{align*}
for some constant $C > 0$, where 
\begin{align*}
    &F(t) := \|\partial_t E \|_\varepsilon^2(t) + \| \nabla  E \|^2(t) + \| \nabla \cdot E \|^2_{\varepsilon - 1}(t) + \| E \|_\sigma^2(t), \\
  &g(t) := \| f_1 \|_\varepsilon^2 + \|\nabla f_0\|^2 + \| \nabla \cdot f_0\|_{\varepsilon - 1}^2 + 2 \| f_0 \|_\sigma^2 +
   \int_0^{t^*} \|j \|^2 \, dt.
\end{align*}
One application of Grönwall's inequality now gives us the result in \eqref{energyestimate}.
\end{proof}

%\subsection{Stability of the semi-discrete problem}

The next corollary follows from the stability estimate for the continuous
problem where all components of the electric field $E$ are replaced with their approximations $E_h$, as well as all other continuous functions are replaced with their discrete analogs.

\begin{corollary}
 Assume that condition
 (\ref{2.3}) on the functions $\varepsilon(x), \sigma(x)$ hold.
 For any $ {t^*}\in \left( 0,T\right)
$ let $\Omega_{{t^*}}= \Omega\times \left( 0, t^* \right)$ and $\Gamma_{t^*} = \Gamma \times (0, t^*).$ Suppose that
  there exists a solution $E_h \in  \textbf{W}_h^E(\Omega)$ of
  the  problem \eqref{eq6} and the approximations of the initial data
  $f_{0,h}$ and $f_{1,h}$  satisfy the regularity conditions
  $f_{1,h}, f_{0,h} \in \textbf{W}_h^E(\Omega)$.
  Then $E_h$ is unique and there exists a constant
  $C = C(\|\varepsilon\|_{C^2(\Omega)}, \| \sigma\|, t^*)$
   such  that
   the following energy estimate is true  for all $\varepsilon \geq 1, \sigma > 0, \varepsilon > \sigma$ in
     \eqref{eq6}:
\end{corollary}

\begin{equation}\label{energyestimatediscrete}
\begin{split}
  \vert\vert\vert E_h \vert\vert \vert^2 (t^*) :=
  & \left\Vert \partial _{t} E_h \right\Vert_{\varepsilon}^{2}(t^*)
  + \left\Vert E_h \right\Vert_{\sigma}^{2}(t^*)
  + \left\Vert \nabla E_h \right\Vert^{2}(t^*)
  + \left\Vert \nabla \cdot E_h \right\Vert _{\varepsilon - 1}^{2}(t^*) \\
  &\leq C \left[ \left\Vert j\right\Vert _{\Omega _{t^*}}^{2}
    +\left\Vert f_{1,h} \right\Vert _{\varepsilon}^{2}
    + \left\Vert \nabla f_{0,h}\right\Vert ^{2}
    + \left\Vert f_{0,h}\right\Vert _{\sigma}^{2}
    + \left\Vert \nabla \cdot f_{0,h}\right\Vert_{\varepsilon - 1}^{2}\right] .
\end{split}
\end{equation}

\section{A priori error estimates}\label{sec:apriori}

  In this section we present an {\em a priori error } estimate for
  the error $e = E(\cdot,t)  - E_h(\cdot,t)$  between the solution $E$
  of the
  model problem \eqref{eq1} and solution $E_h$ of the
  semi-discretized problem  \eqref{eq6}.

  Let 
\begin{equation}\label{error}
    e := E(\cdot,t)  - E_h(\cdot,t) = E - \Pi_h E + \Pi_h E - E_h = \eta + \xi,
\end{equation} 
where $\eta := E - \Pi_h E$, $\xi := \Pi_h E - E_h$. Here, $\Pi_h E\, : \, {\textbf H_E^1(\Omega_T)} \longrightarrow {\mathbf W_h^E(\Omega_T)}$
is an elliptic projection operator  for $E \in H(div, \Omega)$,  see details in \cite{MA,CDE}, such that $\forall {\textbf v}\in{\textbf  W_h^E(\Omega)}$
\begin{equation}\label{ellipticoperator}
    a(\Pi_h E, \textbf{v}) = a(E,\textbf{v}).
\end{equation}

The first part of error, $\eta = E -\Pi_h E$, can be estimated  as follows.

\begin{theorem} \label{ellipticprojection}

   Let  $E \in  {\textbf H_E^2(\Omega_T)}$ bet the solution of the continuous problem \eqref{eq1}.
  Then   
\begin{equation}\label{interpol}
  \begin{split}
    ||\eta ||_{L_2} &\leq C_I(\tau^2 ||D_t^2 E || + h^2 || D_x^2  E ||), \\
     ||\eta ||_{H^1} &\leq C_I(\tau ||D_t^2 E || + h || D_x^2  E ||). 
    \end{split}
\end{equation}
 
  For 
 semi-discretized problem  \eqref{eq6} these estimates
reduces to:
 \begin{equation}\label{interpolest}
   \begin{split}
  ||\eta  ||_{L_2} &\leq C_I h^2 || D_x^2  E ||,\\
  ||\eta ||_{H^1} &\leq C_I h || D_x^2  E ||.
  \end{split}
 \end{equation}

\end{theorem} 

\begin{proof}

  We observe that  using \eqref{ellipticoperator} we can get $\forall v \in \mathbf W_h^E(\Omega_T) $
 \begin{equation}\label{elop1}
   \begin{split} 
     \|\eta \|^2 &= \| E - \Pi_h E \|^2 
     = ( E - \Pi_h E ,E - \Pi_h E )\\
     &= ( E - \Pi_h E ,E - v)
     + ( E - \Pi_h E ,v - \Pi_h E ) \\
     &= ( E - \Pi_h E ,E - v) \leq  \| E - \Pi_h E \| \|E - v \|,
   \end{split}
 \end{equation}
 and thus,
 \begin{equation}\label{elop2}
   \begin{split} 
     \|\eta \| = \| E - \Pi_h E \|  \leq  \|E - v \|.
   \end{split}
 \end{equation}
 Taking $v = E_h^I$  in \eqref{elop2} , where $E_h^I$ is nodal interpolant of $E$, 
 and using
 standard interpolation error estimates \cite{CDE,Johnson,Brenner}
  for
the fully discrete scheme in space and time
 we get
 \begin{equation}\label{elop3}
   \begin{split} 
     \| \eta \|_{L_2}   &\leq C_I(\tau^2 ||D_t^2 E || + h^2 || D_x^2  E ||),\\
       ||\eta ||_{H^1} &\leq C_I(\tau ||D_t^2 E || + h || D_x^2  E ||),
   \end{split}
 \end{equation}
 where $C_I$ are interpolation constants.
  For 
  semi-discretized problem  \eqref{eq6}  terms with $D_t^2 E$ disappear
  and these estimates
 reduces to \eqref{interpolest}.
 
  \end{proof}

In the proof of a priori error estimate 
we use the constant $C$ as a moderate constant which is adjusted throughout the proof, as well as well-posedness of the bilinear form $a(\cdot,\cdot)$. Let us briefly sketch the proof of well-posedness of $a(\cdot,\cdot)$.
We refer to \cite{AB1} for the full details of this proof.

Let us define the linear form as
\begin{equation}
{\mathcal L}( {\textbf v}):= -(j, v).
\end{equation}
We now can     rewrite the equation \eqref{eq6}  as 
\begin{equation}\label{semidiscprob1A}
   \left ( \varepsilon \partial_{tt} E_{h}, {\textbf v} \right ) +
    \left ( \sigma \partial_{t} E_{h}, {\textbf v} \right )
    + a( E_{h},  {\textbf v} )  = {\mathcal L}( {\textbf v}).
\end{equation}

\begin{theorem}[well-posedness of $a(\cdot,\cdot)$ ] \label{Wellposedness1}
 
 Assume that the following condition  holds
 \begin{equation}\label{auxA}
\abs{\nabla\varepsilon}\le \frac{1}{2} \min ( 1/2,  \varepsilon-1). 
\end{equation}
%Then the problem \eqref{eq6} has a unique solution 
 %$\hat{E}_h\in {\textbf  W}_h^E(\Omega_T)$.

 Let
 %us introduce the triple norm 
\begin{equation}\label{TNorm}
%\begin{split}
\vert\vert\vert E_h\vert\vert\vert_{a} ^2:= 
\parallel{E_h}\parallel_{\varepsilon}^2+
\parallel{\nabla E_h}\parallel^2+
\parallel{\nabla \cdot E_h}\parallel_{\varepsilon -1}^2.
\end{equation}
 
 Then for all $ E_h$ and 
${\textbf v} \in {\textbf  W}_h^E(\Omega)$ the discrete bilinear form  $a(\cdot,\cdot)$ is well-posed, or:
  \begin{align}
 a( E_h, E_h)  &  \ge 
 \frac{1}{2} \vert\vert\vert E_h\vert\vert\vert_{a} ^2 
  &&\mbox{(Coercivity of $a$)},  \label{Quasi-coercivity1} \\
 a( E_h, {\textbf v}) &  \le   C_2 \vert\vert\vert E_h\vert\vert\vert_{a} \,
 \vert\vert\vert{\textbf v} \vert\vert\vert_{a}, 
&&\mbox{(Continuity of $a$)}, \label{continuity1}\\
 \abs{{\mathcal L}( {\textbf v})}& \le  C_3
\vert\vert\vert {\textbf v} \vert\vert\vert_{a}, 
&&\mbox{(Continuity of ${\mathcal L} $)}\label{continuity2}. 
  \end{align}
where, $C_i, i=2,3$ are positive constants.
\end{theorem} 

\begin{proof}

  We use  a Lax-Milgram  approach, or we will show that
   the discrete bilinear form
$a( \cdot, \cdot)$ is  coercive, and both $a( \cdot, \cdot)$ and 
   ${\mathcal L}( {\cdot})$ are continuous.

To prove coercivity of $a(\cdot,\cdot)$ we choose
 ${\mathbf{v}}= E_h$ in
\eqref{ah2}  to get 
\begin{equation}\label{bilinearform1A}
\begin{split}
a( E_h,  E_h) &=   (\nabla  E_h, \nabla  E_h)  
+(\nabla (\varepsilon-1)  E_h, \nabla \cdot E_h)
+((\varepsilon-1) \nabla \cdot E_h, \nabla \cdot  E_h) \\                        &=  \parallel{\nabla  E_h}\parallel^2 +
\parallel{\nabla \cdot E_h}\parallel_{\varepsilon -1}^2 
+(\nabla (\varepsilon-1)  E_h, \nabla \cdot E_h).          
\end{split}
\end{equation}
In  \eqref{bilinearform1A} we have used  the equality
\begin{equation}
\begin{split}
  (\nabla \cdot ( (\varepsilon -1 ) E_h), \nabla \cdot  {\textbf v}) =
  ( \nabla (\varepsilon -1)  E_h + (\varepsilon  -1 ) \nabla \cdot E_h, \nabla \cdot  {\textbf v}). 
  \end{split}
  \end{equation} 
Since $\nabla (\varepsilon-1)  =\nabla \varepsilon$ we can use the following estimate  derived in \cite{AB1}:
\begin{equation}\label{CoerciveA1}
\begin{split}
\pm \Big((\nabla \varepsilon) E_h, \nabla \cdot E_h\Big) &\ge 
-\frac 12 \norm{E_h}^2_{\abs{\nabla\varepsilon}}- 
\frac 12 \norm{\nabla \cdot E_h}^2_{\abs{\nabla\varepsilon}},
\end{split}
\end{equation}

Using \eqref{CoerciveA1}  and assumption \eqref{auxA}  we obtain coercivity of
 $a$:
\begin{equation}
\begin{split}
  a( E_h,  E_h) &= \parallel{\nabla E_h}\parallel^2+
\parallel{\nabla \cdot E_h}\parallel_{\varepsilon -1}^2 
+(\nabla (\varepsilon-1)  E_h, \nabla \cdot E_h)\\
&\geq
 \parallel{\nabla E_h}\parallel^2+
\parallel{\nabla \cdot E_h}\parallel_{\varepsilon -1}^2 
-\frac 12 \norm{E_h}^2_{\abs{\nabla\varepsilon}}- 
\frac 12 \norm{\nabla \cdot E_h}^2_{\abs{\nabla\varepsilon}} \geq 
\frac{1}{2} \vert\vert\vert E_h\vert\vert\vert_{a} ^2.
\end{split}
\end{equation}

To prove continuity  of $a(\cdot, \cdot)$, we 
use {\sl Cauchy-Schwarz' inequality} and   estimate \eqref{CoerciveA1}
to obtain   $\forall {\textbf v}\in{\textbf  W_h^E(\Omega)}$  :
%\end{proof} 
%\end{document} 
\begin{equation}\label{Coercivity2}  
\begin{split}
a( E_h, {\textbf v}) &=  
(\nabla  E_h, \nabla {\textbf v}) 
 +(\nabla\cdot ( (\varepsilon-1 ) E_h), \nabla \cdot {\textbf v}) \\
&= (\nabla  E_h, \nabla {\textbf v}) 
+( (\varepsilon-1 )\nabla \cdot E_h, \nabla \cdot {\textbf v}) 
+((\nabla\varepsilon) E_h, \nabla\cdot  {\textbf v}) \\
 &\le 
\parallel{\nabla E_h}\parallel \parallel{\nabla \mathbf{v}}\parallel 
+ \parallel{\nabla \cdot E_h}\parallel _{\varepsilon-1 }
 \parallel \nabla\cdot{\mathbf v}\parallel _{\varepsilon-1 }\\
 &+\parallel{E_h}\parallel_{\abs{\nabla \varepsilon}}
\parallel{\nabla\cdot{\mathbf v}}\parallel _{\abs{\nabla \varepsilon}}
\le
C\vert\vert\vert E_h\vert\vert\vert_{a}
\cdot\vert\vert\vert {\textbf v} \vert\vert\vert_{a}.
\end{split}
\end{equation}

Finally, we can verify continuity of ${\mathcal L}( {\textbf v })$:
%    Likewise, for 
%${\mathcal P}_h f_{0,h}\in L_{2,\varepsilon} (\Omega)\cap 
%L_{2, 1/{s}}(\Gamma_1\cup\Gamma_2)$,  
\begin{equation}\label{coercivity3}
\begin{split}
|{\mathcal L}( {\textbf v })| &= 
|(-j, {\textbf v })| \le  
\parallel{j}\parallel
\parallel{\textbf v}\parallel  \le
 \parallel{j} \parallel
\vert\vert\vert {\textbf v } \vert\vert\vert.
\end{split}
\end{equation}

\end{proof}

\begin{theorem}
  Let  $E \in  {\textbf H_E^2(\Omega_T)}$ solves the continuous problem \eqref{eq1}, and $E_{h} \in {\mathbf W_h^E(\Omega_T)}$ solves the semi-discretized problem \eqref{eq6}.  Assume that $E(t),\partial_{t} E(t),\partial_{tt} E(t) \in H^2(\Omega)$.
  Further assume that the assumptions \eqref{2.3} and \eqref{auxA} on  functions $\varepsilon$ and $\sigma$ hold, as well as $f_0, {f_1} \in [H^1(\Omega)]^3$, ${f_0}_h, {f_1}_h, j_h \in [W_h^E(\Omega)]^3$ and $j \in [L^2(\Omega_T)]^3$. Then  there exists a constant $C(\varepsilon, \sigma)$ such that   for all $t \in [0,T]$ the following a priori error estimates  hold: 

 % [\vert\vert E(\cdot,t) - E_h(\cdot,t) \vert\vert] \leq Ch,
\begin{equation}\label{apriorierror}
\begin{split}
%\norm{\partial_t^{r}e(t)}_\Omega & \le Ch^\zeta\norm{\partial_t^r E(t)}_{H^\zeta(\Omega)} +
%Ch^{\zeta}t^{1-r}\int_0^t\norm{\partial_{ss}^2 E}_{H^\zeta(\Omega)}\, ds \label{ErrorL2}, \quad r=0, 1,\\
  \norm{e(t)} &=  \vert\vert E(\cdot,t) - E_h(\cdot,t) \vert\vert
   \le C_I h^2 (\|  D^2_x E \|_\Omega  + 2d_1 t \int_0^T (\| \partial_{ss} D^2_x E \|_\Omega  +
  \| \partial_{s} D^2_x E \|_\Omega ) \, ds), \\
    \norm{e(t)}_{H^1} &=  \vert\vert E(\cdot,t) - E_h(\cdot,t) \vert\vert_{H^1}
   \le C_I h ( \|  D^2_x E \|_\Omega  + 2 d_1 t \int_0^T (\| \partial_{ss} D^2_x E \|_\Omega  +
   \| \partial_{s} D^2_x E \|_\Omega ) \, ds).
\end{split}
\end{equation}

%where $C := C(\|\varepsilon\|_{C^2(\Omega)}, \| \sigma\|, \|f_0\|_{H^1(\Omega)}, \|f_1\|, \|j\|_{\Omega_T}, \|{f_0}_h\|_{H^1(\Omega)}, \|{f_1}_h\|, \|j\|_{\Omega_T})$.
\end{theorem}
\begin{proof}

Since
    $\eta$ is estimated  via standard interpolation error estimates \eqref{interpol}, we begin to estimate $\xi$.
 We  first note that $\forall {\textbf v}\in{\textbf  W_h^E(\Omega)}$
\begin{align}
     a(\Pi_h E, \textbf{v}) &= a(E,\textbf{v}) \nonumber \\
     &= -(j, \textbf{v}) - (\varepsilon \partial_{tt} E, \textbf{v}) - (\sigma \partial_t E, \textbf{v}) \\
     &+ (\varepsilon \partial_{tt} \Pi_h E, \textbf{v}) + (\sigma \partial_t \Pi_h E, \textbf{v}) \nonumber 
      - (\varepsilon \partial_{tt} \Pi_h E, \textbf{v}) - (\sigma \partial_t \Pi_h E, \textbf{v}). 
\end{align}
Using \eqref{error} the equation above can be rewritten as 
\begin{align}
    (\varepsilon \partial_{tt} \Pi_h E, \textbf{v}) + (\sigma \partial_t \Pi_h E, \textbf{v}) + a(\Pi_h E, \textbf{v}) = -(j,\textbf{v}) - (\varepsilon \partial_{tt} \eta, \textbf{v}) - (\sigma \partial_t \eta, \textbf{v}). 
\end{align}
By subtracting the first equation of \eqref{eq6} from the expression above, while letting $\textbf{v} := \partial_t \xi$ we arrive at
\begin{equation}\label{5.44}
    (\varepsilon \partial_{tt} \xi, \partial_t \xi) + (\sigma \partial_{t} \xi, \partial_t \xi) + a(\xi, \partial_t \xi) =  - (\varepsilon \partial_{tt} \eta, \partial_t \xi) - (\sigma \partial_t \eta, \partial_t \xi).
\end{equation}
Note that $\forall {\textbf v}\in{\textbf  W_h^E(\Omega)}$
\begin{equation*}
    a(\xi, \textbf{v}) = a( \Pi_h E - E_h, \textbf{v}) = a(E - E_h, \textbf{v}) = 0,
\end{equation*}
by the properties of $\Pi_h$ and Galerkin orthogonality.
Using this   we observe that  in \eqref{5.44} the term $ a(\xi, \partial_t \xi)=0$.
Thus, we can estimate a lower bound of the left-hand side of \eqref{5.44} as
\begin{equation}
    (\varepsilon \partial_{tt} \xi, \partial_t \xi) + (\sigma \partial_{t} \xi, \partial_t \xi) + a(\xi, \partial_t \xi) \geq (\varepsilon \partial_{tt} \xi, \partial_t \xi) = \frac{1}{2} \partial_t \| \varepsilon \partial_t \xi \|^2_\Omega \geq \frac{1}{2} \partial_t \| \partial_t \xi \|^2_\Omega,
\end{equation}
where we have used that $(\sigma \partial_{t} \xi, \partial_t \xi) \geq 0$ and $\varepsilon \geq 1, \sigma \geq 0$.

We can also estimate an upper bound for the right-hand side of \eqref{5.44}:
\begin{equation}
    - (\varepsilon \partial_{tt} \eta, \partial_t \xi) - (\sigma \partial_t \eta, \partial_t \xi) \leq (\|\varepsilon \partial_{tt} \eta \|_\Omega + \|\sigma \partial_t \eta \|_\Omega)\|\partial_t \xi\|_\Omega.
\end{equation}
Collecting these two estimates, we have 
\begin{equation}
    \partial_t \| \partial_t \xi \|^2_\Omega \leq 2(\|\varepsilon \partial_{tt} \eta \|_\Omega + \|\sigma \partial_t \eta \|_\Omega)\|\partial_t \xi\|_\Omega. 
\end{equation}
Integrating over $[0, {t^*}]$ where ${t^*} \in [0, T]$ and using conditions
\eqref{2.3}   for functions $\varepsilon, \sigma$ noting that
$\varepsilon > \sigma$
we get
\begin{equation}\label{5.49}
  \begin{split}
     \| \partial_t \xi \|^2_\Omega ({t^*}) &\leq 2\int_0^{t^*} (\|\varepsilon \partial_{ss} \eta \|_\Omega + \|\sigma \partial_s \eta \|_\Omega)\|\partial_s \xi\|_\Omega (s) \, ds  \\
     &\leq 2 d_1 \int_0^{t^*} (\| \partial_{ss} \eta \|_\Omega + \| \partial_s \eta \|_\Omega)(s) \, ds \, \cdot \max_{t' \in [0,T]} \|\partial_s \xi\|_\Omega,  
      \end{split}
\end{equation}
where we have used that $\xi(\cdot, 0)=  \partial_t \xi(\cdot,0) =0 $, and $\sigma \leq \varepsilon \leq d_1$.

Since \eqref{5.49} holds for all ${t^*} \in [0,T]$, we have
\begin{align}\label{5.50}
    \max_{t' \in [0,T]} \| \partial_t \xi \|^2_\Omega &\leq
    2 d_1 \int_0^T (\| \partial_{ss} \eta \|_\Omega + \| \partial_s \eta \|_\Omega)(s) \, ds \, \cdot \max_{t' \in [0,T]} \|\partial_t \xi\|_\Omega, \nonumber \\
    \max_{t' \in [0,T]} \| \partial_t \xi \|_\Omega &\leq
    2 d_1 \int_0^T (\| \partial_{ss} \eta \|_\Omega + \| \partial_s \eta \|_\Omega)(s) \, ds.
\end{align}
Substituting this  estimate into \eqref{5.49}, we finally obtain
\begin{equation}\label{5.51}
    \| \partial_t \xi \|_\Omega \leq
    2 d_1 \int_0^T (\| \partial_{ss} \eta \|_\Omega + \| \partial_s \eta \|_\Omega)(s) \, ds.
\end{equation}

To proceed further we use  estimate \eqref{interpol} to get
\begin{equation}\label{5.53}
   \int_0^{t^*} (\| \partial_{ss} \eta \|_\Omega + \| \partial_s \eta \|_\Omega)(s) \, ds 
  \leq C_I h^2  \int_0^T (\| \partial_{ss} D^2_x E \|_\Omega  +
   \| \partial_{s} D^2_x E \|_\Omega)  \, ds.
\end{equation}
Using \eqref{5.53} in \eqref{5.51} we obtain
\begin{align}\label{5.54}
  \| \partial_t \xi \|_\Omega
     &\leq 2 d_1 C_I h^2  \int_0^T (\| \partial_{ss} D^2_x E \|_\Omega  +
   \| \partial_{s} D^2_x E \|_\Omega ) \, ds.
\end{align}

Further, we observe that
\begin{equation}\label{5.56}
\begin{split}
  \frac d{dt}\norm{\xi}_\Omega^2 =
  2\norm{\xi}_\Omega\frac d{dt}\norm{\xi}_\Omega &= \frac d{dt}\int_\Omega\abs{\xi}^2\, dx \\
&=2\int_\Omega \xi\cdot\xi_t\, dx\le 2\norm{\xi}_\Omega\norm{  \partial_t \xi}_\Omega
\end{split}
\end{equation}
from which it follows that
\begin{equation}\label{5.57}
\begin{split}
 \frac d{dt}\norm{\xi}_\Omega &\le \norm{\partial_t \xi}_\Omega.
\end{split}
\end{equation}

Integrating in time \eqref{5.57} and using \eqref{5.54} yields 
\begin{equation}\label{5.58}
\norm{\xi(t)}_\Omega\le \int_0^t\norm{  \partial_s \xi}_\Omega\, ds
\le  2 d_1 C_I h^2t  \int_0^T (\| \partial_{ss} D^2_x E \|_\Omega  +
   \| \partial_{s} D^2_x E \|_\Omega ) \, ds
\end{equation}

From \eqref{5.58}  and \eqref{interpol}  follows also that
\begin{equation}\label{5.59}
  \norm{\xi(t)}_{H^1}  
  %\int_0^t\norm{  \partial_s \xi}_\Omega\, ds
\le 2 d_1 C_I h\,t \int_0^T (\| \partial_{ss} D^2_x E \|_\Omega  +
   \| \partial_{s} D^2_x E \|_\Omega ) \, ds.
\end{equation}

Summing up  \eqref{interpol},
\eqref{5.58} and \eqref{5.59}
we get the desired  error estimates \eqref{apriorierror},
and the proof is complete.

\end{proof}

\section{Numerical examples}

\label{sec:numex}

% trim crops:
% left, bottom, right and top 
\begin{figure}[h!]
\begin{center}
\begin{tabular}{cc}
 % {\includegraphics[scale=0.45, clip=]{m3}} &
 % {\includegraphics[scale=0.45, clip=]{m4}}
   {\includegraphics[scale=0.3,  trim = 0.0cm 0.0cm  0.0cm 0.0cm, clip=]{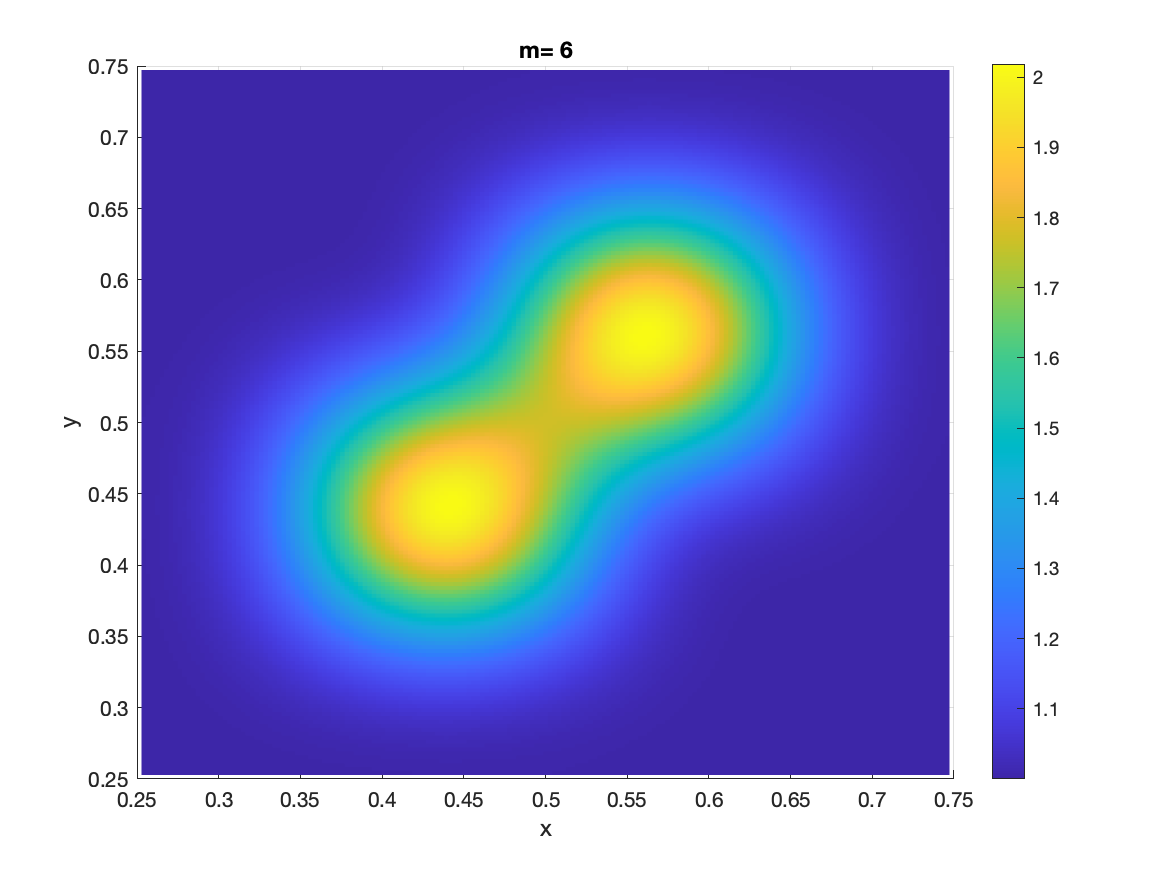}} &
   {\includegraphics[scale=0.3, trim = 0.0cm 0.0cm  0.0cm 0.0cm,  clip=]{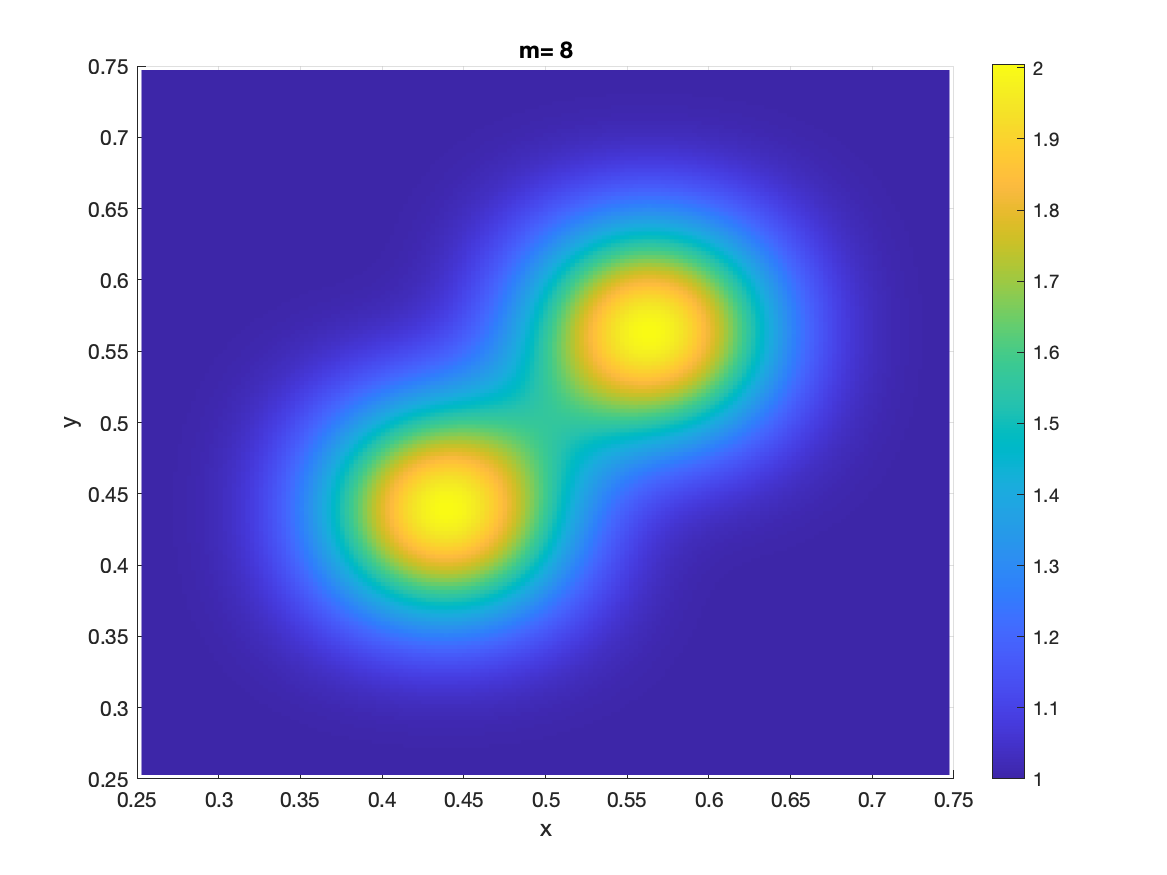}} \\
   a)  $m=6$   & b) $ m=8$  \\
    {\includegraphics[scale=0.3,  trim = 0.0cm 0.0cm 0.0cm 0.0cm, clip=]{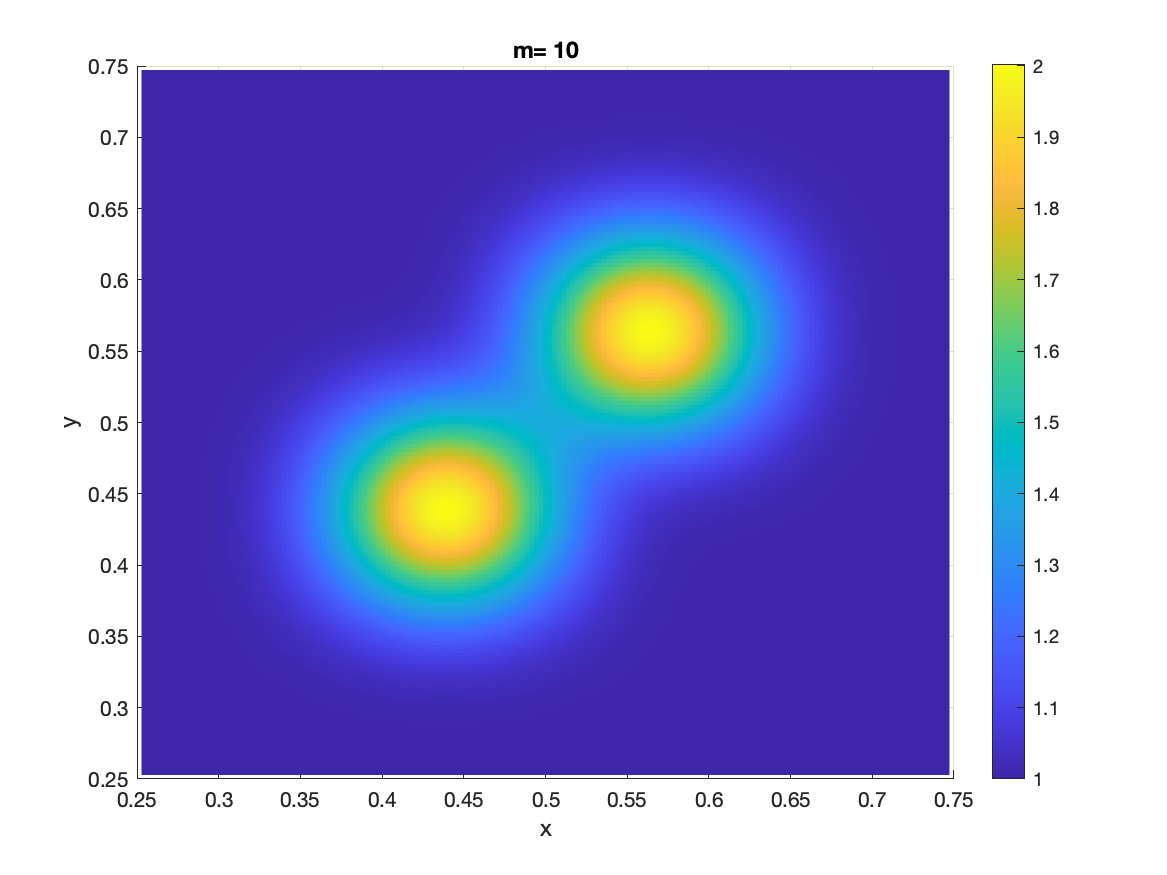}} &
   {\includegraphics[scale=0.3, trim = 0.0cm 0.0cm 0.0cm 0.0cm,  clip=]{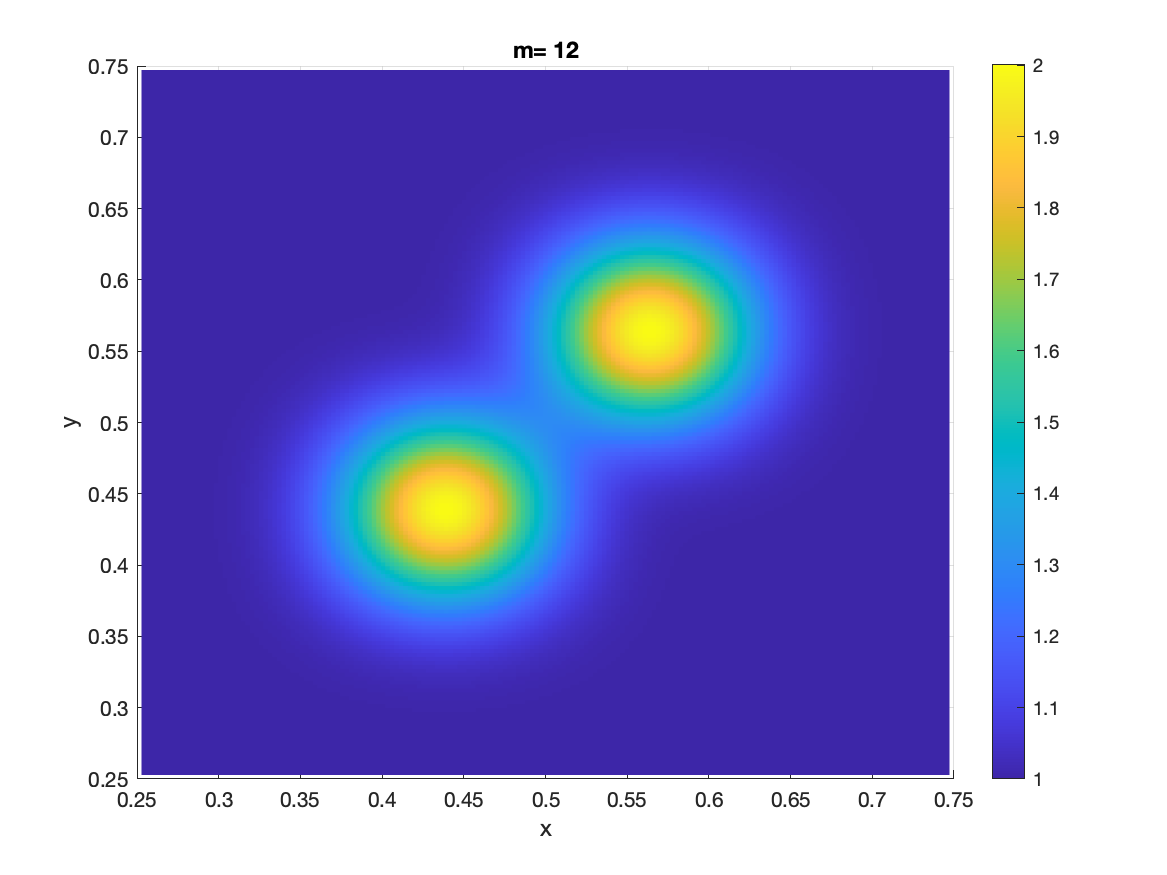}} \\
   c)  $m=10$  & d) $m=12$ 
\end{tabular}
\end{center}
\caption{a) The function $\varepsilon(x,y)$ in the domain $\Omega_1 =[0.25,0.75] \times[0.25,0.75]$  for different
    values of $m$ in \eqref{eps}}.
   \label{fig:F1}
\end{figure}

% trim crops:
% left, bottom, right and top 
\begin{figure}[h!]
\begin{center}
\begin{tabular}{cc}
 % {\includegraphics[scale=0.45, clip=]{m3}} &
 % {\includegraphics[scale=0.45, clip=]{m4}}
   {\includegraphics[scale=0.3,  trim = 0.0cm 0.0cm  0.0cm 0.0cm, clip=]{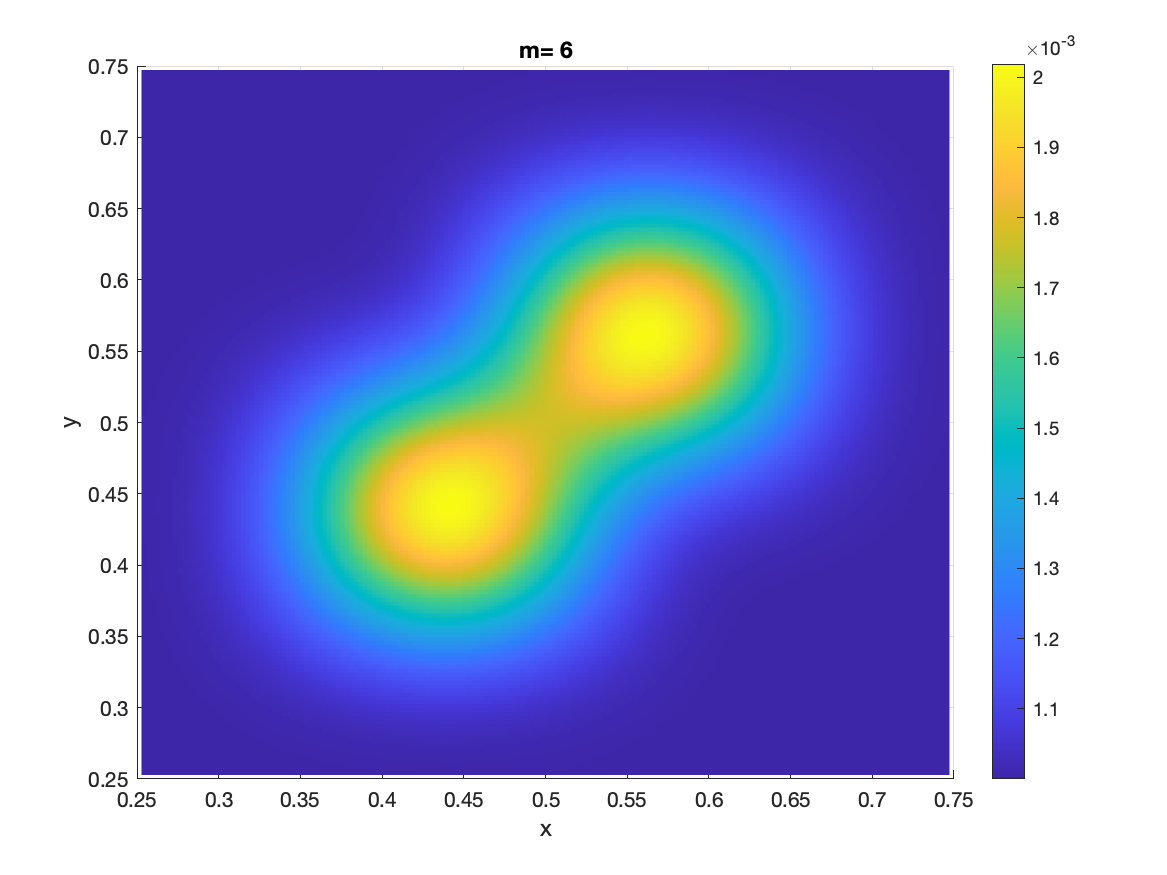}} &
   {\includegraphics[scale=0.3, trim = 0.0cm 0.0cm  0.0cm 0.0cm,  clip=]{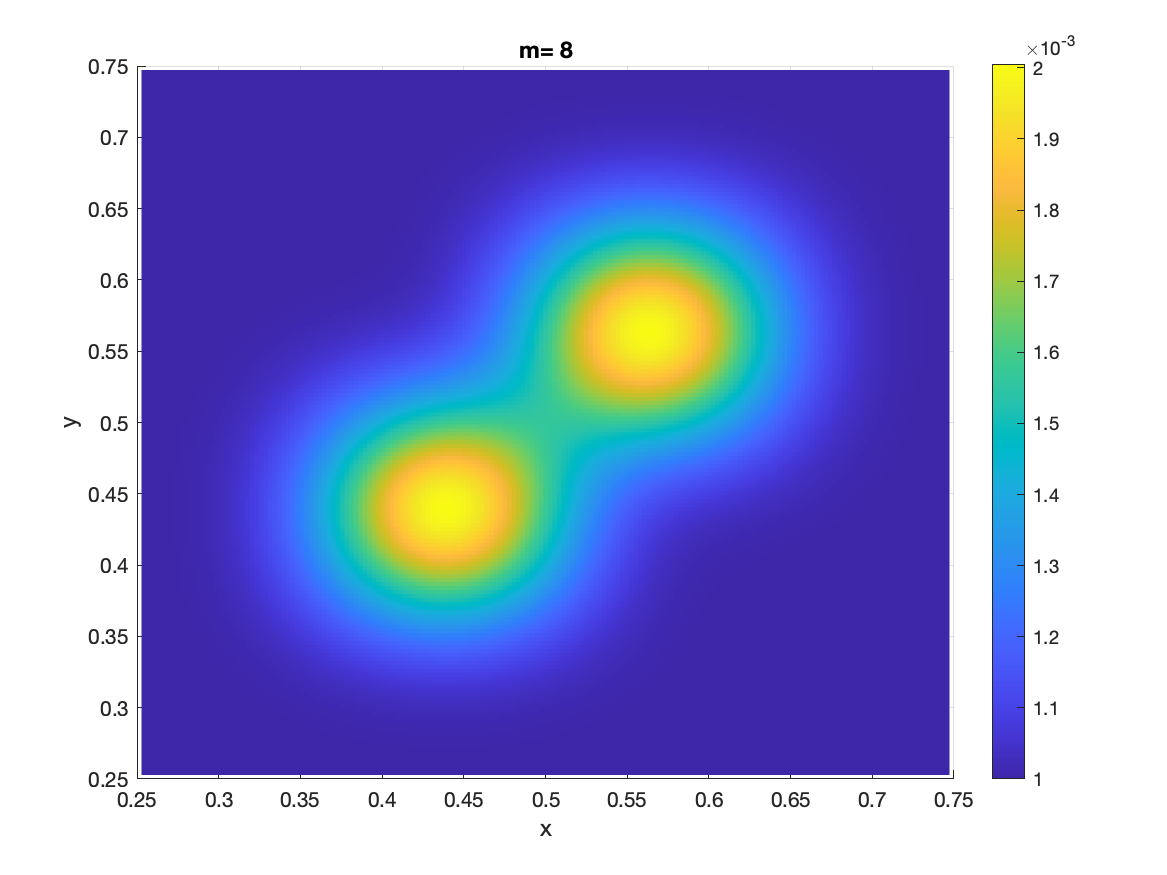}} \\
   a)  $m=6$   & b) $ m=8$  \\
    {\includegraphics[scale=0.3,  trim = 0.0cm 0.0cm 0.0cm 0.0cm, clip=]{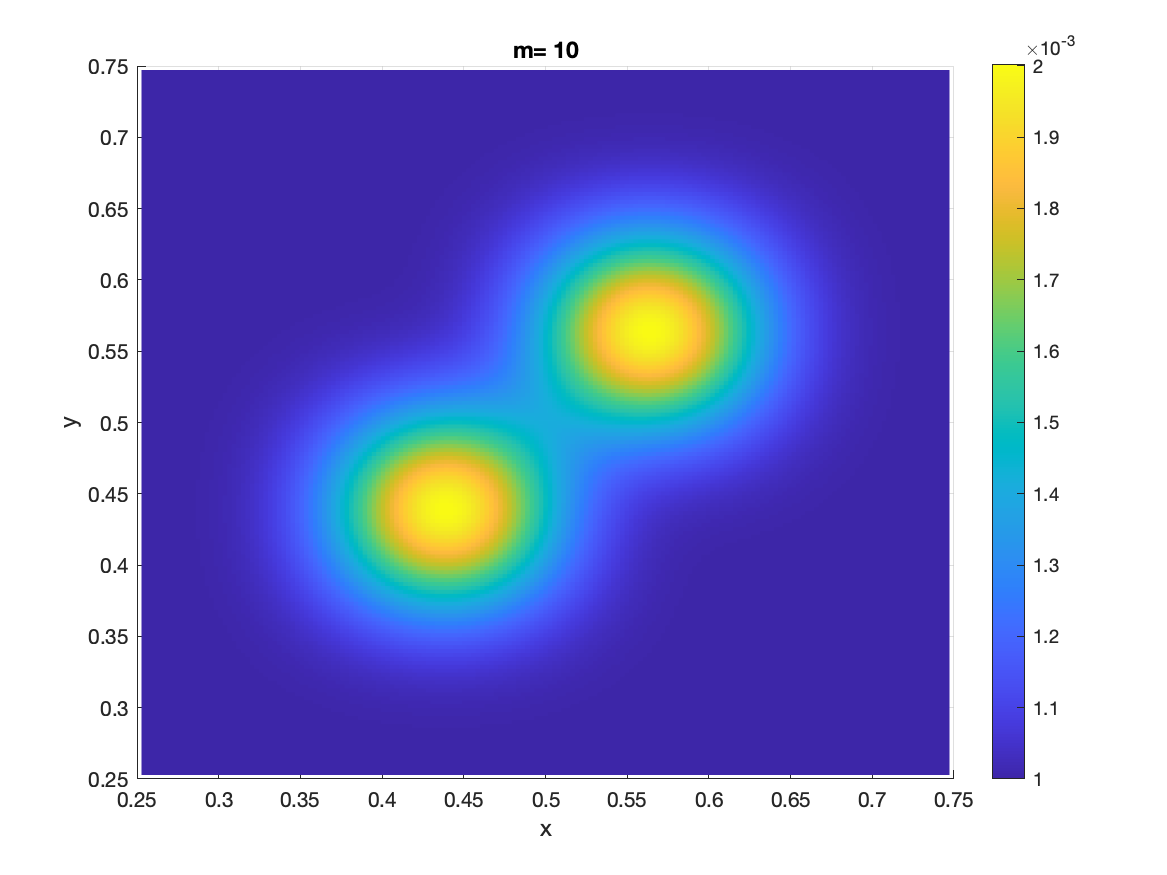}} &
   {\includegraphics[scale=0.3, trim = 0.0cm 0.0cm 0.0cm 0.0cm,  clip=]{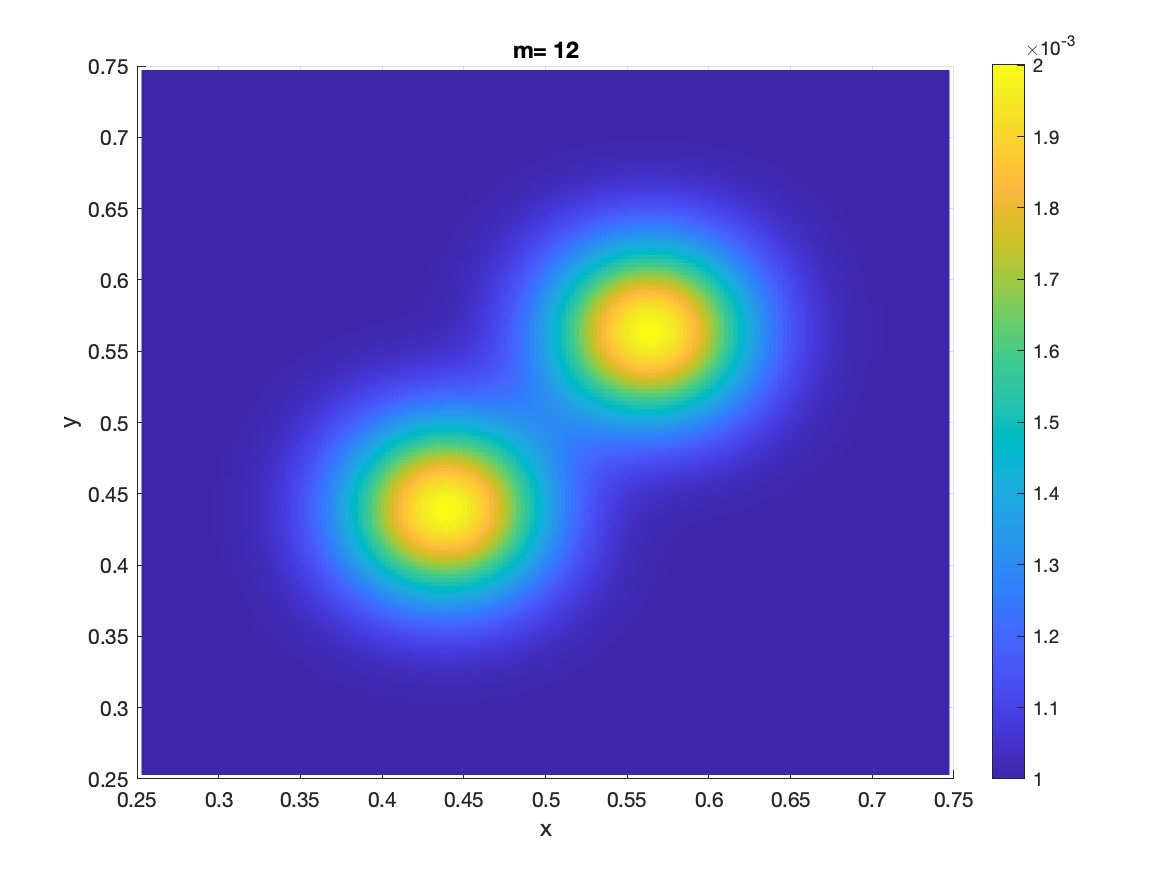}} \\
   c)  $m=10$  & d) $m=12$ 
\end{tabular}
\end{center}
\caption{a) The function $\sigma(x,y)$ in the domain $\Omega_1 =[0.25,0.75] \times[0.25,0.75]$  for different
    values of $m$ in \eqref{eps}}.
   \label{fig:F2}
\end{figure}

% trim crops:
% left, bottom, right and top 
\begin{figure}[h!]
\begin{center}
\begin{tabular}{cc}
 % {\includegraphics[scale=0.45, clip=]{m3}} &
 % {\includegraphics[scale=0.45, clip=]{m4}}
   {\includegraphics[scale=0.4,  trim = 0.0cm 0.0cm  0.0cm 0.0cm, clip=]{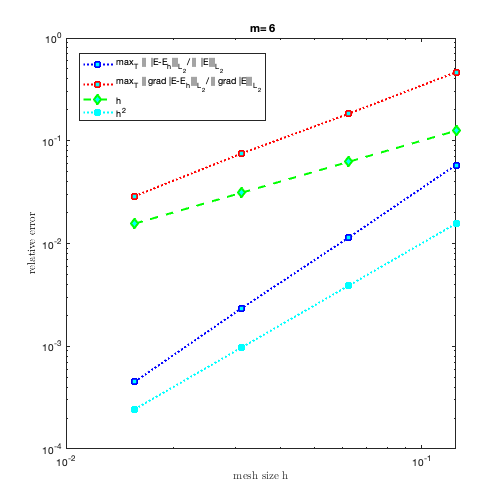}} &
   {\includegraphics[scale=0.4, trim = 0.0cm 0.0cm  0.0cm 0.0cm,  clip=]{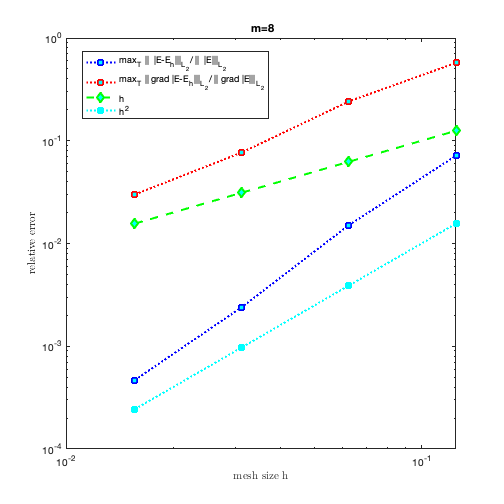}} \\
   a)  $m=6$   & b) $ m=8$  \\
    {\includegraphics[scale=0.4,  trim = 0.0cm 0.0cm 0.0cm 0.0cm, clip=]{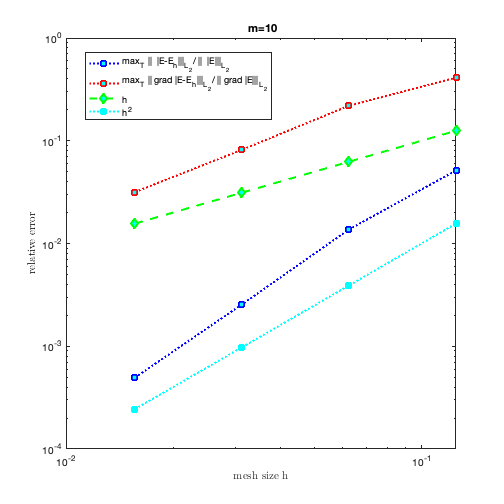}} &
   {\includegraphics[scale=0.4, trim = 0.0cm 0.0cm 0.0cm 0.0cm,  clip=]{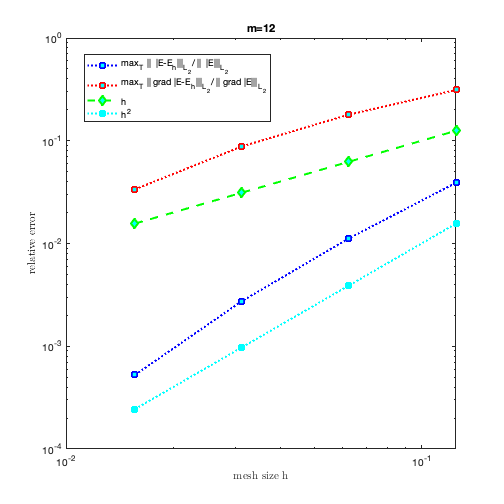}} \\
   c)  $m=10$  & d) $m=12$ 
\end{tabular}
\end{center}
\caption{  Relative errors for different $m$  in \eqref{eps}, \eqref{sigma}}. 
   \label{fig:F3}
\end{figure}

In this section we perform computations which will  confirm  theoretical predictions given in Theorem 4.
All computations are performed in the software package WavES \cite{waves}  using C++/PETSC \cite{petsc}.
The computational domain  $\Omega  \times (0,T)$   is chosen as $\Omega =
[0,1] \times [0,1]$  with $\Omega_1 = [0.25, 0.75] \times [0.25, 0.75]$
such that $ \Omega_1   \subset \Omega$.
 To discretize   the computational domain $\Omega$  
 we denote
 by ${{\mathcal T}_h}_l := \{K\}$ a partition of the domain $\Omega$ into
 triangles $K$ of sizes $h_l= 2^{-l}, l=3,...,6$.
 The explicit  finite element scheme   \eqref{forwfem5}
  derived in \cite{BL2}  was used in computations.
 We have chosen the time step $\tau = 0.0005$  such that the whole  explicit scheme remains stable.

We have used following time-dependent model problem  in  computations:
\begin{equation}\label{E_gauge2}
\begin{split}
  \varepsilon({\mathbf x}) \frac{\partial^2 E({\mathbf x},t)}{\partial t^2} +  
\nabla \times \nabla \times E({\mathbf x},t) +  \sigma({\mathbf x})  \frac{\partial E({\mathbf x},t)}{\partial t}    &= f({\mathbf x},t),  \\
   \nabla \cdot(\varepsilon E)({\mathbf x},t) &= 0, \\
   E({\mathbf x},0) = 0, ~~~\frac{\partial E}{\partial t}({\mathbf x},0) &= 0, \\
    E|_{\Gamma}=0.
\end{split}
\end{equation}
The source data
 $f({\mathbf x},t), ~ {\mathbf x}:=(x,y) \in \mathbb{R}^{2},~ t \in [0,0.25]$ is
computed by knowing the exact solution
\begin{equation}\label{exactsol}
  \begin{split}
    E_1({\mathbf x},t) &=    \frac{t^2}{\varepsilon}  \pi \sin^2 \pi x  \cos \pi y  \sin \pi y, \\
   E_2({\mathbf x},t) &=  - \frac{t^2}{\varepsilon}   \pi   \sin^2 \pi y \cos \pi x   \sin \pi x   
  \end{split}
\end{equation}
of the problem  \eqref{E_gauge2}.

 In the model problem
 \eqref{E_gauge2}
 the function $\varepsilon(x,y)$ is defined   as
\begin{equation}\label{eps}
  \varepsilon(x,y)= \left \{
  \begin{array}{lll}
    1 &+ (\sin \pi (2x-0.375))^m \cdot (\sin \pi (2y-0.375))^m  & \\
      &+  (\sin \pi (2x-0.625))^m \cdot (\sin \pi (2y-0.625))^m & \textrm{in  $\Omega_1$}, \\
    1 &  &\textrm{otherwise,}
  \end{array}
  \right.
\end{equation}
 and the function $\sigma(x,y)$    as
\begin{equation}\label{sigma}
  \sigma(x,y)= \left \{
  \begin{array}{lll}
    0.001(1 &+ (\sin \pi (2x-0.375))^m \cdot (\sin \pi (2y-0.375))^m & \\
            &+  (\sin \pi (2x-0.625))^m \cdot (\sin \pi (2y-0.625))^m) & \textrm{in  $\Omega_1$}, \\
    0 &  & \textrm{otherwise}
  \end{array}
  \right.
\end{equation}
Figures \ref{fig:F1},  \ref{fig:F2}    show the functions $\varepsilon$  and $\sigma$, respectively,
for   different $m = 6,8,10,12$ in  \eqref{eps}, \eqref{sigma}  which were used in computations.

Relative errors  $\ \Theta^{(1)} , \Theta^{(2)} $
are computed at the time moment $t=0.25$ as
\begin{align}
\ \Theta^{(1)}  &=  \frac{\|  \hat{E}  -  \hat{E}_h \|_{L_2}}{\| \hat{E} \|_{L_2}},\qquad 
\mbox{and} \\
\Theta^{(2)}  &=  \frac{\| \nabla (\hat{E}- \hat{E}_h) \|_{L_2}}{\| \nabla \hat{E} \|_{L_2}},  
\end{align}
 in  $L_2$- and 
 $H^1$-norms,
 respectively.
Here, $\hat{E} = (\hat{E}_{1}, \hat{E}_{2})$ is the exact solution
given by \eqref{exactsol}, and $\hat{E}_h = (\hat{E}_{1h},
\hat{E}_{2h})$ is the computed solution. We note also that
\begin{equation}\label{AbsE}
  \abs{\hat{E}}:=\sqrt{\hat{E}_{1}^2 + \hat{E}_{2}^2}\qquad
   \abs{\hat{E}_h}:=\sqrt{\hat{E}_{1h}^2 + \hat{E}_{2h}^2}. 
\end{equation}

Figures \ref{fig:F3} present convergence results of explicit finite element scheme
    \eqref{forwfem5}  
for the functions $\varepsilon$ and $\sigma$   defined by \eqref{eps},
\eqref{sigma}, respectively, for different values of $m = 6, 8, 10, 12$.
Table 1, Table 2, Table 3 and Table 4 present
relative errors   $ \Theta^{(j)}_l, ~j=1,2 $  and convergence rates  $ r^{(j)}_l, ~j=1,2 $ in the $L_2$-norm and in the $H^1$-norm for
mesh sizes $h_l= 2^{-l}, l=3,...,6$,
for different values of  $m = 6,8,10,12$ in \eqref{eps},\eqref{sigma}.
We note that chosen values of $m$ satisfy the regularity assumptions on the exact solution, see details in \cite{BR1}.

We used following  expressions to compute convergence
rates $r^{(1)}$ and $r^{(2)}$   presented in these figures  and tables:
\begin{equation}
  \begin{split}
  r^{(1)} &= \frac{ \left|\log \left( \frac{   \Theta^{(1)}_{l} }{   \Theta^{(1)} _{l+1}} \right) \right|}{|\log(2)|},\\
  r^{(2)} &= \frac{ \left|\log \left( \frac{  \Theta^{(2)}_{l}}{ \Theta^{(2)} _{l+1}} \right) \right|}{|\log(2)|},
  \end{split}
  \end{equation}
where $ \Theta^{(j)}_l, ~j=1,2 $ are computed relative norms $ \Theta^{(j)}, ~j=1,2 $
on the mesh ${\mathcal T}_h$ with the  mesh sizes $h_l= 2^{-l}, l=3,...,5$.

\begin{table}[h!] 
\center
\begin{tabular}{ | l | l  |  l  |  l | l  |  l  |  l | l  |  l |  }
\hline
$l$ &  $nel$  & $nno$ &  $ \ \Theta^{(1)} $  &  $\frac{  \ \Theta^{(1)} _l}{  \ \Theta^{(1)} _{l+1}}$ & $r^{(1)}$ & $ \Theta^{(2)} $  & $\frac{  \Theta^{(2)} _l}{ \Theta^{(2)} _{l+1}}$ & $r^{(2)}$ \\
\hline 
$3$ & $128$ & $81$    &  $0.058066$  &     -            &     -     &  $0.464524$  &  -  &  -  \\
$4$ & $512$ & $289$   &  $0.011481$  &      $5.057543$  &   $2.34$  &  $0.183696$  &  $2.528771$  &$1.34$  \\
$5$ & $2048$ & $1089$ &  $0.002355$  &      $4.875048$  &   $2.29$  &  $0.075362$  &  $2.437524$  &$1.29 $  \\
$6$ & $8192$ & $4225$ &  $0.000453$  &      $5.202624$  &   $2.38$  &  $0.028971$  &  $2.601312$  &$1.38  $  \\
\hline
\end{tabular}
\caption{Relative errors   $ \Theta^{(j)}_l, ~j=1,2 $  and convergence rates  $ r^{(j)}_l, ~j=1,2 $ in the $L_2$-norm and in the $H^1$-norm for
  mesh sizes $h_l= 2^{-l}, l=3,...,6$,  for $m=6$ in \eqref{eps},\eqref{sigma}.}
\label{testm6}
\end{table}

\begin{table}[h!] 
\center
\begin{tabular}{ | l | l  |  l  |  l | l  |  l  |  l | l  |  l |  }
\hline
$l$ &  $nel$  & $nno$ &  $ \ \Theta^{(1)} $  &  $\frac{  \ \Theta^{(1)} _l}{  \ \Theta^{(1)} _{l+1}}$ & $r^{(1)}$ & $ \Theta^{(2)} $  & $\frac{  \Theta^{(2)} _l}{ \Theta^{(2)} _{l+1}}$ & $r^{(2)}$ \\
\hline 
$3$ & $128$ & $81$    &  $0.071545$  &     -            &     -     &  $0.572362$  &  -  &  -  \\
$4$ & $512$ & $289$   &  $0.015110$  &      $4.735050$  &   $2.24$  &  $0.241756$  &  $2.367525$  &$1.24$  \\
$5$ & $2048$ & $1089$ &  $0.002406$  &      $6.280222$  &   $2.65$  &  $0.076989$  &  $3.140111$  &$1.65 $  \\
$6$ & $8192$ & $4225$ &  $0.000469$  &      $5.130590$  &   $2.36$  &  $0.030012$  &  $2.565295$  &$1.36  $  \\
\hline
\end{tabular}
\caption{Relative errors   $ \Theta^{(j)}_l, ~j=1,2 $  and convergence rates  $ r^{(j)}_l, ~j=1,2 $ in the $L_2$-norm and in the $H^1$-norm for
  mesh sizes $h_l= 2^{-l}, l=3,...,6$,  for $m=8$ in \eqref{eps},  \eqref{sigma}.}
\label{testm8}
\end{table}

\begin{table}[h!] 
\center
\begin{tabular}{ | l | l  |  l  |  l | l  |  l  |  l | l  |  l |  }
\hline
$l$ &  $nel$  & $nno$ &  $ \ \Theta^{(1)} $  &  $\frac{  \ \Theta^{(1)} _l}{  \ \Theta^{(1)} _{l+1}}$ & $r^{(1)}$ & $ \Theta^{(2)} $  & $\frac{  \Theta^{(2)} _l}{ \Theta^{(2)} _{l+1}}$ & $r^{(2)}$ \\
\hline 
$3$ & $128$ & $81$    &  $0.051348$  &     -            &     -     &  $0.410785$  &  -  &  -  \\
$4$ & $512$ & $289$   &  $0.013703$  &      $3.747278$  &   $1.91$  &  $0.219245$  &  $1.873639$  &$0.91$  \\
$5$ & $2048$ & $1089$ &  $0.002553$  &      $5.367863$  &   $2.42$  &  $0.081688$  &  $2.683932$  &$1.42 $  \\
$6$ & $8192$ & $4225$ &  $0.000495$  &      $5.156936$  &   $2.37$  &  $0.031681$  &  $2.578468$  &$1.37$  \\
\hline
\end{tabular}
\caption{Relative errors     $ \Theta^{(j)}_l, ~j=1,2 $  and convergence rates  $ r^{(j)}_l, ~j=1,2 $ in the $L_2$-norm and in the $H^1$-norm for
  mesh sizes $h_l= 2^{-l}, l=3,...,6$,  for $m=10$ in \eqref{eps},  \eqref{sigma}.}
\label{testm10}
\end{table}

\begin{table}[h!] 
\center
\begin{tabular}{ | l | l  |  l  |  l | l  |  l  |  l | l  |  l |  }
\hline
$l$ &  $nel$  & $nno$ &  $ \ \Theta^{(1)} $  &  $\frac{  \ \Theta^{(1)} _l}{  \ \Theta^{(1)} _{l+1}}$ & $r^{(1)}$ & $ \Theta^{(2)} $  & $\frac{  \Theta^{(2)} _l}{ \Theta^{(2)}_{l+1}}$ & $r^{(2)}$ \\
\hline 
$3$ & $128$ & $81$    &  $0.038995$  &     -            &     -     &  $0.311959$  &  -  &  -  \\
$4$ & $512$ & $289$   &  $0.011230$  &      $3.472240$  &   $1.80$  &  $0.179688$  &  $1.736128$  &$0.80$  \\
$5$ & $2048$ & $1089$ &  $0.002753$  &      $4.078874$  &   $2.03$  &  $0.088106$  &  $2.039437$  &$1.03 $  \\
$6$ & $8192$ & $4225$ &  $0.000526$  &      $5.238828$  &   $2.39$  &  $0.033636$  &  $2.619414$  &$1.39$  \\
\hline
\end{tabular}
\caption{Relative errors    $ \Theta^{(j)}_l, ~j=1,2 $  and convergence rates  $ r^{(j)}_l, ~j=1,2 $
  in the $L_2$-norm and in the $H^1$-norm for
  mesh sizes $h_l= 2^{-l}, l=3,...,6$,  for $m=12$ in \eqref{eps},  \eqref{sigma}.}
\label{testm12}
\end{table}

Using  Figures  \ref{fig:F3}   and tables we observe that  the explicit finite element scheme
derived in   \cite{BL2} behaves like a first order method in
 in $H^1(\Omega)$-norm and second
 order method in $L_2(\Omega)$-norm.
Therefore, these results  confirm theoretical analytic   estimates derived in Theorem 4.

\section{Conclusions}

\label{sec:conclusion}

This paper presents stability and convergence analysis for the finite element method for
stabilized time-dependent
Maxwell's equations in conductive nonmagnetic media developed in  \cite{BL2}.
We present analysis for a specific case when
the dielectric permittivity  and conductivity functions
have
 a constant value in a
 boundary neighborhood.
 
 In the theoretical part of the paper  we derived energy norm stability estimates
for the continuous and discrete solutions of the model problem, as well as 
 a priori error bounds in
 the gradient dependent, weighted  norms.  Our numerical
 computations confirm theoretical predictions and show that our method
 behaves like a first order method in $H^1(\Omega)$-norm and second
 order method in $L_2(\Omega)$-norm.

\vspace{0.5cm}

\noindent \textbf{\underline{Acknowledgment}} The research of both authors
is supported by the Swedish Research Council grant VR 2018-03661.

\rule{2mm}{2mm}

%\newpage 

\end{document}